\documentclass{amsart}
\usepackage{amsfonts}
\usepackage{amsmath,amssymb}
\usepackage{amsthm}
\usepackage{amscd}
\usepackage{graphics}
\usepackage{graphicx}

\theoremstyle{remark}{
\newtheorem{Def}{{\rm Definition}}

\newtheorem{Rem}{{\rm Remark}}

}
\theoremstyle{plain}
{

\newtheorem{Prop}{Proposition}
\newtheorem{Thm}{Theorem}

}

\begin{document}
\title[Precise shapes of certain height functions on products of spheres]{Height functions on products of spheres and associated Reeb digraphs and level sets}
\author{Naoki kitazawa}

\address{Osaka Central Advanced Mathematical Institute (OCAMI) \\
3-3-138 Sugimoto, Sumiyoshi-ku Osaka 558-8585
TEL: +81-6-6605-3103
}
\email{naokikitazawa.formath@gmail.com}
\urladdr{https://naokikitazawa.github.io/NaokiKitazawa.html}
\maketitle
\begin{abstract}
{\it Height functions} are fundamental and important objects and tools in mathematics, especially in geometry such as differential topology and differential geometry and some related singularity theory of differentiable maps.

Our interest, especially interest of the author, lies in obtaining explicit lists of such functions. Recently, as information on so-called higher degrees, he is also interested in their naturally defined 1st derivatives. 

We consider natural maps on products of spheres which are variants of specific cases of so-called {\it moment maps} on {\it toric symplectic} manifolds. We also generalize cases of the canonical projections of the unit spheres. This is a further result on related previous study of the author. We use {\it Reeb graphs}, graphs being natural quotient spaces of manifolds of the domains of nice functions such as Morse-Bott functions, and consisting of connected components of level sets. They are fundamental tools and objects since the last century. 

\end{abstract}
\section{Introduction.}
\label{sec:1}
\subsection{Introductory exposition on height functions.}
\label{subsec:1.1}
A {\it height} function $c:X \subset {\mathbb{R}}^{m_0} \rightarrow {\mathbb{R}}$ is a real-valued smooth function on an $m$-dimensional smooth submanifold $X$ of the $m_0$-dimensional Euclidean space ${\mathbb{R}}^{m_0}$ represented as the restriction of the canonical projection ${\pi}_{m_0,1}:{\mathbb{R}}^{m_0} \rightarrow \mathbb{R}$, where ${\pi}_{k,k_1}:{\mathbb{R}}^{k} \rightarrow {\mathbb{R}}^{k_1}$ denotes the canonical projection, mapping the point $x=(x_1,x_2) \in {\mathbb{R}}^{k_1} \times {\mathbb{R}}^{k-k_1}={\mathbb{R}}^{k}$ to $x_1 \in {\mathbb{R}}^{k_1}$ ($k \geq k_1 \geq 1$).

The following are also important.
\begin{itemize}
	\item ${\mathbb{R}}^k$ (, where $\mathbb{R}$ in the case $k=1$,) is a Riemannian manifold equipped with the standard Euclidean metric. This is also seen as the $k$-dimensional real affine space.  
	\item The $m$-dimensional unit sphere $S^{m}:=\{(x_1, \ldots x_m,x_{m+1}) \in {\mathbb{R}}^{m_0}={\mathbb{R}}^{m+1} \mid {\Sigma}_{j=1}^{m+1} {x_j}^2=1\}$ is of simplest cases. This is generalized to the canonical projection obtained by the restriction of ${\pi}_{m+1,k_1}$ with $1 \leq k_1 \leq m+1$, whose image is the unit $k_1$-dimensional unit disk $D^{k_1} \subset {\mathbb{R}}^{k_1}$, which is a compact and connected submanifold of ${\mathbb{R}}^{k_1}$ whose boundary is $S^{k_1-1}$.  
	\item 
We can consider a differential $dc$ as a bundle morphism between the {\it tangent bundle} $TX$ over $X$ into the tangent bundle $T\mathbb{R}$ over $\mathbb{R}$. Mapping a tangent vector of length $1$ along (so-called) {\it positive gradient flow} by this and composing the square, we can define the {\it canonical 1st derivative} $c^{\prime,{\rm P}}$. This is a smooth function whose values are $0 \leq c(x) \leq 1$ (for a function $c$ of a certain class of tame smooth functions).
\end{itemize}

Such functions are fundamental and important geometric objects and tools in mathematics, especially geometry such as differential geometry and differential topology.

\subsection{Our interest, especially main interest of the author, related to height functions, list of explicit cases.}
\label{subsec:1.2}
\subsubsection{Our interest on height functions and exposition on related history including one on traditional singularity theory.}
We are, and especially the author is, interested in obtaining explicit cases of height functions and their canonical 1st derivatives. This also comes from some singularity theory of differential functions and maps, applicable to such geometry.
It is important in such theory to construct explicit functions and maps and making lists of them.

Recently, the author has launched related studies \cite{kitazawa7, kitazawa8}, where we do not assume non-trivial knowledge or arguments there. There, the height functions of the unit spheres and the canonical projections of the unit spheres, presented in the first subsection, are respected. Related studies, or construction of smooth maps whose images are given regions in the Euclidean spaces, are pioneered in \cite{kitazawa3}, followed by the author himself in \cite{kitazawa4, kitazawa5, kitazawa6}. This respects so-called {\it special generic} maps, generalizing the canonical projections of the unit spheres and the so-called {\it Morse} functions in Reeb's sphere theorem. For special generic maps, see \cite{burletderham, furuyaporto, saeki1} and for traditional theory of Morse functions, which is still strong and important, see \cite{milnor1, milnor2}. For singularity theory of differentiable functions and maps, see a textbook \cite{golubitskyguillemin} for example.

Morse-Bott functions are important as generalized Morse functions. For this, see also \cite{banyagahurtubise, bott}. 
A certain specific class of so-called {\it moment maps} on so-called {\it symplectic and toric manifolds} is generalized to smooth maps locally represented as canonical products of two height functions of the unit spheres on $S^{m_1}$ and $S^{m_2}$, respectively. The images of these maps are rectangles or more generally, so-called convex polytopes.
Such maps are regarded to be variants of special generic maps. For related traditional theory, see   
\cite{buchstaberpanov, delzant}. Related to our theory, recent explicit studies on singularity theory of smooth maps into ${\mathbb{R}}^2$ such as \cite{kobayashi, kobayashiyamamoto} are important and deforming and visualizing these maps and the spaces of the domains are discussed explicitly, through calculations and observations. 
By composing the canonical projection to the straight line, we also have Morse-Bott functions in considerable cases here.  

Note also that these studies of Kobayashi and Yamamoto are regarded to be recent studies respecting traditional studies on higher dimensional versions of Morse functions such as \cite{thom, whitney}. We can also say that the studies on special generic maps are also located in such a story. 
\subsubsection{Our construction of functions, maps and manifolds, representations via Reeb graphs, and our main result.}

We consider natural reconstruction of smooth ({\it real algebraic}) maps onto rectangles in the plane. For this, we can explain via \cite{kitazawa4, kitazawa8}, where we do not need related knowledge. We review in a self-contained way and we only discuss specific cases from \cite{kitazawa8}. We can respect \cite{kobayashi, kobayashiyamamoto} and we respect some of them, where we discuss in a self-contained way. 

We use {\it Reeb graphs} of (nice) real-valued functions such as Morse-Bott functions. They have been fundamental and strong tools in understanding manifolds via nice functions, since \cite{reeb} and the establishment of traditional theory of Morse functions, in the midst of the 20th century. They are natural quotient spaces of the manifolds of the domains and regarded to be the space of all connected components of all preimages ({\it level sets}) of the functions. 

They are naturally graphs.

They are also important objects and related topological studies and combinatorial one are, incredibly, new. One of such studies is, reconstructing nice smooth functions with given Reeb graphs. This is pioneered in \cite{sharko} in 2006, by Sharko. \cite{masumotosaeki,michalak} are related studies in the 2010s. Since 2020s, Gelbukh has contributed to this by papers \cite{gelbukh1, gelbukh2, gelbukh3, gelbukh5, gelbukh6} and a preprint \cite{gelbukh8}, with a related survey preprint \cite{gelbukh7}. There local construction of smooth functions represented by branched covering, Morse functions, or so on, are important. Contribution of the author is, reconstruction respecting not only Reeb graphs but also shapes of level sets, such as \cite{kitazawa1, kitazawa3} and real algebraic one pioneered in \cite{kitazawa2}, followed by the author himself in \cite{kitazawa4, kitazawa5, kitazawa6} for example.

Our main result is roughly, summarized as Theorem \ref{thm:0}. Hereafter, the restriction of a map $c:X \rightarrow Y$ to $Z \subset X$ is denoted by $c {\mid}_Z$.
\setcounter{Thm}{-1}
\begin{Thm}
	\label{thm:0}
	Onto quadrilateral $D_{\rm R} \subset {\mathbb{R}}^2$
	surrounded by two parallel straight lines and additional two straight lines, reconstruction of canonical smooth {\rm (}real algebraic{\rm )} maps $f_{m_1,m_2}:X \rightarrow {\mathbb{R}}^2$ on manifolds $S^{m_1} \times S^{m_2}$ of dimension $m:=m_1+m_2$ embedded as smooth submanifolds $X \subset {\mathbb{R}}^{m_0}$ with $m_0:=m+3$ {\rm (}or $m_0:=m+2${\rm )} are considered with $m_1, m_2 \geq 1$.

What follows is our main result. 

In this way, height functions are obtained as the restrictions ${\pi}_{m+2,1} {\mid}_{X}$ and represented by ${\pi}_{2,1} \circ f_{m_1,m_2}$.

Their canonical 1st derivatives are studied. More precisely, their critical sets and fundamental differential topological properties are studied{\rm :} their Reeb graphs, their level sets, and images of the manifolds $X$ by ${\pi}_{m+2,2} {\mid}_{X}$ are explicitly presented in several quadrilaterals $D_{\rm R} \subset {\mathbb{R}}^2$.
\end{Thm}
\subsection{The content of the present paper.}
\label{subsec:1.3}
The next section is for additional preliminaries and some terminologies, notions, and notation we have presented in the present section or ones we have not presented rigorously. We have presented some rigorously and some roughly in this section and roughly presented exposition is rigorously presented in the second section. This is no problem. The third section is for our main result. 

Theorem \ref{thm:0} is rigorously presented again, as Theorem \ref{thm:3} and \ref{thm:4}, with their proof.  Before presenting the result, we also review related previous studies on reconstruction of nice real algebraic maps onto regions in ${\mathbb{R}}^2$ of the author, in a self-contained way, and we present Theorems \ref{thm:1} and \ref{thm:2}. In previous preprints such as \cite{kitazawa8}, essentially more general cases have been studied and our present cases are specific cases. We do not assume non-trivial knowledge related to these previous studies.   
\section{Additional preliminaries.}
\label{sec:2}
\subsection{Elementary terminologies, notions, and notation.}
\label{subsec:2.1}
For a topological space $X$ we can define the ({\it topological}) {\it dimension} as a non-negative integer uniquely, we use $\dim X$ for this. Topological manifolds and CW complexes are of this class. 

For a subspace $Y$ in a topological space $X$, we use ${\overline{Y}}^X$ for its closure there.

For a differentiable manifold $X$ (of the class $C^r$, with "$C^{\infty}$" being for the smooth cases), we use $T_p X$ for the tangent space of $X$ at $p \in X$. This is a real vector space of dimension $\dim X$ and we can topologize $TX={\bigcup}_{p \in X} T_p X$ and make it a differentiable manifold (resp. of the class $C^r$) canonically. More precisely, $TX$ is the tangent bundle of $X$. It is a real vector bundle over $X$ and can be topologized and made a differentiable manifold of dimension $2\dim X$ (of the class $C^r$) via local transformations associated canonically with (the so-called Jacobi) matrices ${(\frac{\partial x_{j_1}}{\partial x_{j_2}})}_{1 \leq j_1 \leq \dim X, 1 \leq j_2 \leq \dim X}$ with local coordinates.

For a differentiable map $c:X \rightarrow Y$ between differentiable manifolds $X$ and $Y$, the differential ${dc}:TX \rightarrow TY$ is a bundle map and $p \in X$ is a {\it singular} point of $c$ if the rank of the linear map ${dc}_p:T_p X \rightarrow T_{c(p)} Y$ defined at $p$ is smaller than the minimum of $\{\dim X,\dim Y\}$. We use $S(c)$ for the set of all singular points of the differentiable map $c$, and we also call this the {\it singular set} of $c$. We may use "{\it critical}" instead of "singular" if $\dim Y \leq 1$.

A smooth map which is also a homeomorphism and which has no singular point is a {\it diffeomorphism}. 
For a diffeomorphism between a manifold $X$ onto itself, we call it a {\it diffeomorphism} on $X$.
In all smooth manifolds, we can define the equivalence relation by the existence of diffeomorphisms. We can say two manifolds are {\it diffeomorphic} if they belong to the same equivalence class.

We use $0:=(0, \ldots) \in {\mathbb{R}}^k$ for the origin of ${\mathbb{R}}^k$.

An {\it affine subspace} of the {\it $k$-dimensional real affine space} ${\mathbb{R}}^k$ is the space represented as the zero set of a real polynomial map each component of which is of degree $0$ or $1$. An {\it affine isomorphism} between an affine subspace of ${\mathbb{R}}^{m_1}$ and an affine subspace of ${\mathbb{R}}^{m_2}$ is a diffeomorphism which is obtained by restricting some map obtained by the composition of a linear map from ${\mathbb{R}}^{m_1}$ to ${\mathbb{R}}^{m_2}$ with a parallel transformation. This defines an equivalence relation on the family of all affine subspaces (of some real affine spaces).

Consider a case $X \subset {\mathbb{R}}^{m_0}$ of an $m$-dimensional smooth submanifold $X$ of ${\mathbb{R}}^{m_0}$ with no boundary and with $m<m_0$ and a union of connected components of the zero set of a smooth map $f_{X,{\mathbb{R}}^{m_0},{\mathbb{R}}^{m-m_0}}:{\mathbb{R}}^{m_0} \rightarrow {\mathbb{R}}^{m_0-m}$ with $f_{X,{\mathbb{R}}^{m_0},{\mathbb{R}}^{m_0-m}} {\mid}_{X}$ containing no singular point. If $f_{X,{\mathbb{R}}^{m_0},{\mathbb{R}}^{m_0-m}}:{\mathbb{R}}^{m_0} \rightarrow {\mathbb{R}}^{m_0-m}$ is a real polynomial, then $X$ is said to be {\it real algebraic}. A {\it real algebraic} map means the composition of the canonical embedding into ${\mathbb{R}}^{m_0}$ with the canonical projection ${\pi}_{m_0,k_1}$.

\subsection{Morse-Bott functions.}
\label{subsec:2.2}
\begin{Def}
	\label{def:1}
A smooth function $c:X \rightarrow Y$ with $Y:=\mathbb{R}, S^1$ on a Riemannian manifold satisfying the following is {\it Morse-Bott}.
\begin{itemize}
\item It contains no critical point on the boundary $\partial X \subset X$ of $X$.
\item The critical set $S(c)$ is a disjoint union of smooth and connected submanifolds $C_{\lambda}(c)$ with no boundary each of which is labeled by $\lambda \in \Lambda$.
\item For each point $p \in C_{\lambda}(c)$, we can have the normal vector space $N_{p,C_{\lambda}(c)} \subset T_p X$ uniquely, as the space of all elements in $T_p X$ orthogonal to any element $v_{p,C_{\lambda}} \in T_p C_{\lambda}$. If we restrict the so-called {\it Hessian}, a bilinear form on $T_p X$ defined by the symmetric matrix ${(\frac{{\partial}^2 c}{\partial x_{j_1} \partial x_{j_2}})}_{1 \leq j_1 \leq \dim X, 1 \leq j_2 \leq \dim X}$ canonically, to $N_{p,C_{\lambda}(c)} \subset T_p X$, then it is non-degenerate for some local coordinate. 
\end{itemize}
This is also {\it Morse} if and only if the connected submanifolds are all discrete points in $S(c)$.
\end{Def}

\begin{Prop}
\label{prop:1}
	In Definition \ref{def:1}, for each $p \in C_{\lambda}(c)$ above, the sign $i(c,p)=i(c,\lambda)$ of the Hessian is defined uniquely, as the number of negative components in the uniquely and canonically defined diagonal matrix of degree $\dim X-\dim C_{\lambda}(c)$. We call this the {\it index} of $p$ and the index of $C_{\lambda}(c)$ for $c$.
\end{Prop}

\subsection{Reeb spaces (graphs) of smooth real-valued functions.}
\label{subsec:2.3}
In the present paper, we assume fundamental knowledge of graphs (and so-called {\it digraphs}). They are $0$- or $1$-dimensional CW complex which is locally finite and whose closure of $1$-cells (in the graphs) are homeomorphic to $D^1$.
$1$-cells (resp $0$-cells) are called {\it edges} (resp. {\it vertices}) of the graphs.

Hereafter, a preimage $c^{-1}(y)$ of a real-valued function $c$ is a {\it level set} of $c$ and if the set of the domain is a topological space, then each connected component of it is a {\it contour} of $c$. 
The {\it Reeb space} ({\it Reeb graph}) $R_c:=X/{\sim}_c$ of a smooth function $c:X \rightarrow \mathbb{R}$ is the quotient space of the manifold $X$, with the following equivalence relation ${\sim}_c$ on $X$. We can define this by $p_1 {\sim}_c p_2$ if and only if $p_1$ and $p_2$ are in a same contour of $c$. We have the quotient map $q_c:X \rightarrow R_c$ with the unique continuous map $\bar{c}:R_c \rightarrow \mathbb{R}$ yielding $c=\bar{c} \circ q_c$. In certain nice situations, $\dim R_c$ is defined to be $0 \leq \dim R_c \leq 1$, and $R_c$ is homotopic or homeomorphic to a graph. In a certain more specific case, we make this a graph by defining its vertices as points $v$ with ${q_c}^{-1}(v)$ having some critical points of $c$. 
Here, the contour ${q_c}^{-1}(v)$ having some critical points of $c$ is said to be a {\it critical} contour.
This is the {\it Reeb graph} of $c$ and we can orient each edge $e$ of it, which departs from a vertex $v_{e,1}$ of it and enter another vertex $v_{e,2}$ of it. This is a digraph and we call it the {\it Reeb digraph} $\overrightarrow{R_c}$ of $c$.

 For smooth real-valued functions $c$ on closed manifolds with $c(S(c))$ being finite, $R_c$ are the Reeb graphs of $c$, due to \cite[Theorem 3.1]{saeki2}. For related general topological theory of Reeb spaces, see \cite{gelbukh4, saeki3} and the survey preprint \cite{gelbukh7} for example, in addition. We do not need related general theory here.  
 
Specific cases for Reeb spaces are first discussed in \cite{reeb} and see also \cite{izar}. They are for Morse functions on compact manifolds. For Morse-Bott functions on compact manifolds, see \cite{martinezalfaromezasarmientooliveira} for example.
 \section{On our main result.}
\label{sec:3}
 \subsection{The canonical 1st derivative $c^{\prime,{\rm P}}$ of a height function $c:={\pi}_{m_0,1}:X \rightarrow \mathbb{R}$ on an $m$-dimensional smooth submanifold $X$ with no boundary of ${\mathbb{R}}^{m_0}$.}
 \label{subsec:3.1}
We review the {\it canonical 1st derivative $c^{\prime,{\rm P}}$} of a height function $c:={\pi}_{m_0,1}:X \rightarrow \mathbb{R}$ on an $m$-dimensional smooth submanifold $X$ with no boundary of ${\mathbb{R}}^{m_0}$, respecting the preprint \cite{kitazawa8}, in a self-contained way. Note that the phrase "canonical" and the notation "$c^{\prime,{\rm P}}$" are first presented, here.
 
For each point $p \in X$ which is not a critical point of $c$, we can have the unique tangent vector $v_{p,c,+} \in T_pX \subset T_p {\mathbb{R}}^{m_0}$ as follows.
\begin{itemize}
\item This is orthogonal to all tangent vectors of $T_pX$ mapped to the zero vector of $T_{c(p)} \mathbb{R}$ by the differential $d{\pi}_{m_0,1}$.
\item Its length is $1$.
\item It is is mapped to a tangent vector of $T_{c(p)} \mathbb{R}$  by $d{\pi}_{m_0,1}$ being not the zero vector, and which is oriented naturally in the positive direction.
\end{itemize}
By composing this correspondence with the differential $d{{\rm id}_{\mathbb{R}}}$ of the identity function ${\rm id}_{\mathbb{R}}$ on $\mathbb{R}$ and the square, we have a smooth real-valued function on $X-S(c)$. Due to fundamental theory from Riemannian geometry especially so-called ({\it geodesic}) ({\it positive}) {\it gradient flow}, by assigning $0$ to each critical point of $c$, we have a non-negative smooth function $c^{\prime,{\rm P}}:X \rightarrow \mathbb{R}$ with its zero set being $S(c)$ and with value at each point must be smaller than or equal to $1$. This is the {\it canonical 1st derivative $c^{\prime,{\rm P}}$}.

\subsection{The canonical projection ${\pi}_{m+1,1} {\mid}_{S^m}$ revisited.}
\label{subsec:3.2}
We review fundamental properties of ${\pi}_{m+1,1} {\mid}_{S^m}$ important in our study.

 First, the Reeb digraph of ${\pi}_{m+1,1} {\mid}_{S^m}$ ($m \geq 2$) is a digraph with exactly one edge and two vertices. In addition, that of ${\pi}_{2,1}$ is a digraph with exactly two edges departing from a vertex and entering another vertex, and exactly two vertices, which are presented here.  

We discuss ${{\pi}_{m+1,1} {\mid}_{S^m}}^{\prime,{\rm P}}$, in a self-contained way, respecting \cite{kitazawa8}.
Here we use {\it great circles}, passing through $(\pm 1,0) \in \mathbb{R} \times {\mathbb{R}}^m$ and parameterized by $\{0\} \times S^{m-1} \subset S^m$ (with nice symmetry). Each great circle is naturally regarded to be $S^1$ and each point is represented as $(\cos \theta,\sin \theta)$
 by the naturally defined angle $\theta$. For each point, a tangent vector is represented by the form $(-\sin \theta,\cos \theta)$. From this,  ${{\pi}_{m+1,1} {\mid}_{S^m}}^{\prime,{\rm P}}$ is represented as a function mapping each point represented by the pair of a point in  $p_{S^{m-1}} \in \{0\} \times S^{m-1} \subset S^m$ and $(\cos \theta,\sin \theta)$ to ${\sin}^2 \theta$. The 2nd derivative
$\frac{{\partial}^2 ({\sin}^2 \theta)}{\partial \theta \partial \theta}$
 of the function $\Theta(p_{S^{m-1}},\theta)={\sin}^2 \theta$  is $2\cos (2\theta)$. Due to this calculation and the symmetry by  $\{0\} \times S^{m-1} \subset S^m$, ${{\pi}_{m+1,1} {\mid}_{S^m}}^{\prime,{\rm P}}$ is shown to be Morse-Bott.

\begin{Prop}
\label{prop:2}
${{\pi}_{m+1,1} {\mid}_{S^m}}^{\prime,{\rm P}}$ is Morse-Bott. Its critical set is the disjoint union of ${{{\pi}_{m+1,1} {\mid}_{S^m}}^{\prime,{\rm P}}}^{-1}(0)=\{(-1,0 \ldots),(1,0 \ldots)\}$ and ${{{\pi}_{m+1,1} {\mid}_{S^m}}^{\prime,{\rm P}}}^{-1}(1)=\{0\} \times S^{m-1}$. The index of $(\pm 1,0 \ldots)$ for the function ${{\pi}_{m+1,1} {\mid}_{S^m}}^{\prime,{\rm P}}$ is $0$ with each point there being isolated in $S^m$, and  the index of each connected component of $\{0\} \times S^{m-1}$ for the function ${{\pi}_{m+1,1} {\mid}_{S^m}}^{\prime,{\rm P}}$  is $1$.

In the case $m \geq 2$, its Reeb digraph $\overrightarrow{R_{{{\pi}_{m+1,1} {\mid}_{S^m}}^{\prime,{\rm P}}}}$  is with exactly three vertices $v_{0,1}$, $v_{0,2}$, and $v_1$, and two edges $e_{0,j}$, each of which departs from $v_{0,j}$ and enters $v_1$.

In the case $m=1$, its Reeb digraph $\overrightarrow{R_{{{\pi}_{m+1,1} {\mid}_{S^m}}^{\prime,{\rm P}}}}$  is with exactly four vertices $v_{0,1}$, $v_{0,2}$, $v_{1,1}$, and $v_{1,2}$, and exactly four edges $e_{j_1,j_2}$, each of which departs from $v_{0,j_1}$ and enters $v_{1,j_2}$. For this case $m=1$, we give our related remark on  \cite{kitazawa5}. The induced function $\bar{{{\pi}_{2,1} {\mid}_{S^1}}^{\prime,{\rm P}}}$, the unique function with the relation ${{\pi}_{2,1} {\mid}_{S^1}}^{\prime,{\rm P}}=\bar{{{\pi}_{2,1} {\mid}_{S^1}}^{\prime,{\rm P}}} \circ q_{{{\pi}_{2,1} {\mid}_{S^1}}^{\prime,{\rm P}}}$, cannot be represented as the composition of an embedding into ${\mathbb{R}}^2$ with  ${\pi}_{2,1}$.
	\end{Prop}
We present $\overrightarrow{R_{{{\pi}_{m+1,1} {\mid}_{S^m}}^{\prime,{\rm P}}}}$ {\rm (}$m \geq 2${\rm }) and $\overrightarrow{R_{{{\pi}_{2,1} {\mid}_{S^1}}^{\prime,{\rm P}}}}$  in Figure \ref{fig:1}.
\begin{figure}
	\includegraphics[width=40mm, height=40mm]{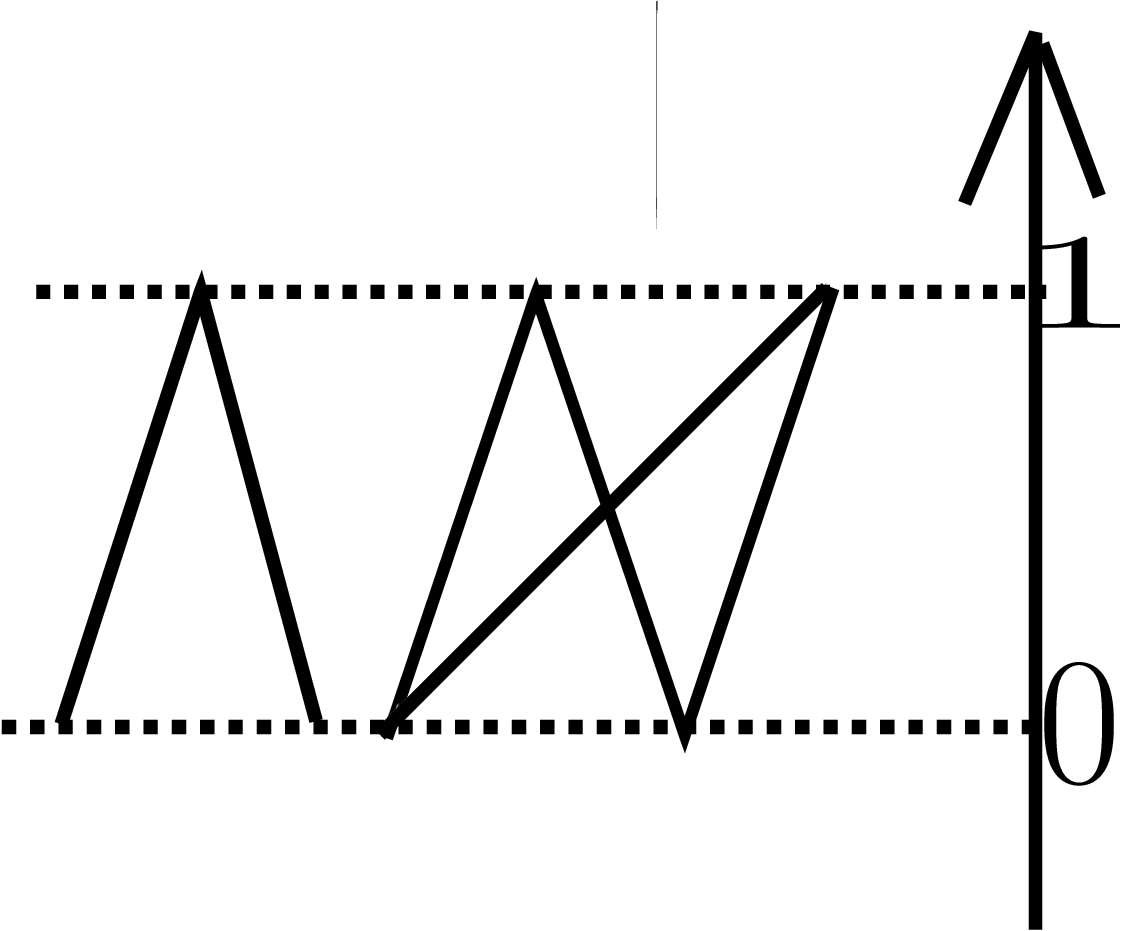}
	\caption{The left figure shows $\overrightarrow{R_{{{\pi}_{m+1,1} {\mid}_{S^m}}^{\prime,{\rm P}}}}$ ($m \geq 2$). The right figure shows $\overrightarrow{R_{{{\pi}_{2,1} {\mid}_{S^1}}^{\prime,{\rm P}}}}$ .}
	\label{fig:1}
\end{figure}
\subsection{Natural real algebraic maps onto quadrilateral in ${\mathbb{R}}^2$}
\label{subsec:3.3}
A {\it quadrilateral} means a non-empty open, connected and bounded subset of ${\mathbb{R}}^2$ surrounded by $4$-distinct straight lines in (, or $1$-dimensional affine subspaces of) ${\mathbb{R}}^2$, where in its closure of ${\mathbb{R}}^2$, distinct two lines of these four intersect in a one-element set, or do not intersect.

Theorem \ref{thm:1} is important and an essentially more general case is considered in \cite{kitazawa8}, first. We review its proof.

\begin{Thm}
\label{thm:1}
Let $a>0$.
	Let $S_{a,-}:=\{(-a,x_2) \mid x_2 \in \mathbb{R}\}$ and $S_{a,+}:=\{(a,x_2) \mid x_2 \in \mathbb{R}\}$.
Let $S_{{\rm l},-1}:=\{(x_1,x_2) \mid x_2=a_{-1,1}x_1+a_{-1,2}\}$, $S_{{\rm l},1}=\{(x_1,x_2) \mid x_2=a_{1,1}x_1+a_{1,2}\}$ and assume that $a_{-1,1}x_1+a_{-1,2}<a_{1,1}x_1+a_{1,2}$ holds for $-a \leq x \leq a$.

We can argue as follows.
\begin{enumerate}
\item \label{thm:1.1}
We have the unique quadrilateral $D_{S_a,a_{-1},a_{1}}:=\{(x_1,x_2) \mid -a \leq x_1 \leq a, a_{-1,1}x_1+a_{-1,2}<a_{1,1}x_1+a_{1,2} \} \subset {\mathbb{R}}^2$.
\item \label{thm:1.2}
 A natural smooth family $S_{{\rm l},t}:=\{(x_1,x_2) \mid x_2=(-\frac{1}{2}t+\frac{1}{2})(a_{-1,1}x_1+a_{-1,2})+(\frac{1}{2}t+\frac{1}{2})(a_{1,1}x_1+a_{1,2})\}$ of straight lines parameterized by $-1 \leq t \leq 1$ is obtained. If we restrict each straight line of the family to $\{x_1 \mid -a \leq x_1 \leq a\}$, then for distinct $t=t_1$ and $t=t_2$, the resulting segments are disjoint. 
\item \label{thm:1.3}
 For positive integers $m_1$ and $m_2$ and $m:=m_1+m_2$, we have an $m$-dimensional real algebraic manifold $X_{D_{S_a,a_{-1},a_{1}},m_1,m_2}:=\{(x_1,x_2,t,{(y_{1,j_1})}_{j_1=1}^{m_1},{(y_{2,j_2})}_{j_2=1}^{m_2}) \mid a^2-{x_1}^2-{\Sigma}_{j_1}^{m_1} {y_{1,j_1}}^2=0, 1-t^2-{\Sigma}_{j_2=1}^{m_2} {y_{2,j_2}}^2=0, x_2=(-\frac{1}{2}t+\frac{1}{2})(a_{-1,1}x_1+a_{-1,2})+(\frac{1}{2}t+\frac{1}{2})(a_{1,1}x_1+a_{1,2})\} \subset {\mathbb{R}}^{m+3}$. This is mapped onto the closure ${\overline{D_{S_a,a_{-1},a_{1}}}}^{{\mathbb{R}}^2}$. 
\item \label{thm:1.4}
Furthermore, the restriction ${\pi}_{m+3,1} {\mid}_{X_{D_{S_a,a_{-1},a_{1}},m_1,m_2}}$ is a Morse-Bott function the image of whose singular set is $\{-a,a\}$ and the critical contour $({\pi}_{m+3,1} {\mid}_{X_{D_{S_a,a_{-1},a_{1}},m_1,m_2}})^{-1}(\pm a)$ of which is diffeomorphic to $S^{m_2}$. 
\end{enumerate}
\end{Thm}
\begin{proof}
We review original proofs in a self-contained way.

We put $f_{{X_{D_{S_a,a_{-1},a_{1}},m_1,m_2}},1}:=a^2-{x_1}^2-{\Sigma}_{j_1}^{m_1} {y_{1,j_1}}^2$, $f_{{X_{D_{S_a,a_{-1},a_{1}},m_1,m_2}},2}:=1-t^2-{\Sigma}_{j_2=1}^{m_2} {y_{2,j_2}}^2$, and $f_{{X_{D_{S_a,a_{-1},a_{1}},m_1,m_2}},3}:=x_2-(-\frac{1}{2}t+\frac{1}{2})(a_{-1,1}x_1+a_{-1,2})-(\frac{1}{2}t+\frac{1}{2})(a_{1,1}x_1+a_{1,2})$. 

We apply implicit function theorem for these three defining functions for the manifold $X_{D_{S_a,a_{-1},a_{1}},m_1,m_2}$, at each point $(x_1,x_2,t,{(y_{1,j_1})}_{j_1=1}^{m_1},{(y_{2,j_2})}_{j_2=1}^{m_2}) \in X_{D_{S_a,a_{-1},a_{1}},m_1,m_2}$. \\
\ \\

Case 1-1 The case $(x_1,x_2) \in D_{S_a,a_{-1},a_{1}}$. \\

The value of the derivative $\frac{\partial f_{{X_{D_{S_a,a_{-1},a_{1}},m_1,m_2}},1}}{\partial y_{1,j_{1,0}}}$ is non-zero for some $y_{1,j_{1,0}}$ and those of the derivatives $\frac{\partial f_{{X_{D_{S_a,a_{-1},a_{1}},m_1,m_2}},1}}{\partial y_{2,j_{2}}}$ and $\frac{\partial f_{{X_{D_{S_a,a_{-1},a_{1}},m_1,m_2}},1}}{\partial t}$  are $0$. 

The value of the derivative $\frac{\partial f_{{X_{D_{S_a,a_{-1},a_{1}},m_1,m_2}},2}}{\partial y_{2,j_{2,0}}}$ is non-zero for some $y_{2,j_{2,0}}$ and those of the derivatives $\frac{\partial f_{{X_{D_{S_a,a_{-1},a_{1}},m_1,m_2}},2}}{\partial y_{1,j_{1}}}$ and $\frac{\partial f_{{X_{D_{S_a,a_{-1},a_{1}},m_1,m_2}},2}}{\partial t}$  are $0$. 

The value of the derivative $\frac{\partial f_{{X_{D_{S_a,a_{-1},a_{1}},m_1,m_2}},3}}{\partial t}$ is non-zero and those of the derivatives $\frac{\partial f_{{X_{D_{S_a,a_{-1},a_{1}},m_1,m_2}},3}}{\partial y_{i,j_{i}}}$  are $0$. 

We apply implicit function theorem for some $y_{1,j_{1,0}}$, some $y_{2,j_{2,0}}$, and $t$. \\

\ \\ 
Case 1-2 The case $(x_1,x_2) \in {\overline{D_{S_a,a_{-1},a_{1}}}}^{{\mathbb{R}}^2} \bigcap S_{a,+} \sqcup S_{a,-}$ with $(x_1,x_2) \notin  S_{{\rm l},-1} \sqcup S_{{\rm l},1}$. \\

The value of the derivative $\frac{\partial f_{{X_{D_{S_a,a_{-1},a_{1}},m_1,m_2}},1}}{\partial x_1}$ is non-zero and those of the derivatives $\frac{\partial f_{{X_{D_{S_a,a_{-1},a_{1}},m_1,m_2}},1}}{\partial y_{i,j_{i}}}$ and $\frac{\partial f_{{X_{D_{S_a,a_{-1},a_{1}},m_1,m_2}},1}}{\partial t}$  are $0$. 

The value of the derivative $\frac{\partial f_{{X_{D_{S_a,a_{-1},a_{1}},m_1,m_2}},2}}{\partial y_{2,j_{2,0}}}$ is non-zero for some $y_{2,j_{2,0}}$ and those of the derivatives $\frac{\partial f_{{X_{D_{S_a,a_{-1},a_{1}},m_1,m_2}},2}}{\partial x_1}$ and $\frac{\partial f_{{X_{D_{S_a,a_{-1},a_{1}},m_1,m_2}},2}}{\partial t}$  are $0$. 

The value of the derivative $\frac{\partial f_{{X_{D_{S_a,a_{-1},a_{1}},m_1,m_2}},3}}{\partial t}$ is non-zero and that of the derivative $\frac{\partial f_{{X_{D_{S_a,a_{-1},a_{1}},m_1,m_2}},3}}{\partial y_{i,j_{i}}}$  is $0$. 

We apply implicit function theorem for some  $y_{2,j_{2,0}}$, $x_1$, and $t$. \\

\ \\
Case 1-3  The case $(x_1,x_2) \in {\overline{D_{S_a,a_{-1},a_{1}}}}^{{\mathbb{R}}^2} \bigcap S_{{\rm l},-1} \sqcup S_{{\rm l},1}$ with $(x_1,x_2) \notin S_{a,+} \sqcup S_{a,-}$. \\

The value of the derivative $\frac{\partial f_{{X_{D_{S_a,a_{-1},a_{1}},m_1,m_2}},1}}{\partial y_{1,j_{1,0}}}$ is non-zero  for some $y_{1,j_{1,0}}$ and that of the derivative $\frac{\partial f_{{X_{D_{S_a,a_{-1},a_{1}},m_1,m_2}},1}}{\partial t}$  is $0$. 

The value of the derivative $\frac{\partial f_{{X_{D_{S_a,a_{-1},a_{1}},m_1,m_2}},2}}{\partial t}$ is non-zero and those of the derivatives $\frac{\partial f_{{X_{D_{S_a,a_{-1},a_{1}},m_1,m_2}},2}}{\partial y_{i,j_{i}}}$ and $\frac{\partial f_{{X_{D_{S_a,a_{-1},a_{1}},m_1,m_2}},2}}{\partial x_2}$ are $0$. 

The values of the derivatives $\frac{\partial f_{{X_{D_{S_a,a_{-1},a_{1}},m_1,m_2}},3}}{\partial x_2}$ and $\frac{\partial f_{{X_{D_{S_a,a_{-1},a_{1}},m_1,m_2}},3}}{\partial t}$ are non-zero and that of the derivative $\frac{\partial f_{{X_{D_{S_a,a_{-1},a_{1}},m_1,m_2}},3}}{\partial y_{i,j_{i}}}$ is $0$. 

We apply implicit function theorem for some  $y_{1,j_{1,0}}$, $x_2$, and $t$. \\

\ \\
Case 1-4 The remaining case, $(x_1,x_2) \in {\overline{D_{S_a,a_{-1},a_{1}}}}^{{\mathbb{R}}^2} \bigcap (S_{a,+} \sqcup S_{a,-}) \bigcap (S_{{\rm l},-1} \sqcup S_{{\rm l},1})$. \\

The value of the derivative $\frac{\partial f_{{X_{D_{S_a,a_{-1},a_{1}},m_1,m_2}},1}}{\partial x_1}$ is non-zero and those of the derivatives $\frac{\partial f_{{X_{D_{S_a,a_{-1},a_{1}},m_1,m_2}},1}}{\partial x_2}$  and $\frac{\partial f_{{X_{D_{S_a,a_{-1},a_{1}},m_1,m_2}},1}}{\partial t}$ are $0$. 

The value of the derivative $\frac{\partial f_{{X_{D_{S_a,a_{-1},a_{1}},m_1,m_2}},2}}{\partial t}$ is non-zero and those of the derivatives $\frac{\partial f_{{X_{D_{S_a,a_{-1},a_{1}},m_1,m_2}},2}}{\partial x_1}$ and $\frac{\partial f_{{X_{D_{S_a,a_{-1},a_{1}},m_1,m_2}},2}}{\partial x_2}$  are $0$. 

The values of the derivatives $\frac{\partial f_{{X_{D_{S_a,a_{-1},a_{1}},m_1,m_2}},3}}{\partial x_2}$ and $\frac{\partial f_{{X_{D_{S_a,a_{-1},a_{1}},m_1,m_2}},3}}{\partial t}$ are non-zero.

We apply implicit function theorem for $x_1$, $x_2$, and $t$. \\

This completes the proof.

\end{proof}
Note that in the original preprint, more general situations are considered. We do not discuss such situations.

For {\it rectangles}, or cases where the four lines are divided into pairs of mutually parallel lines, refer to a kind of related studies such as \cite{kobayashi, kobayashiyamamoto}. There local maps around a point $(\pm 1,\pm 1)$ in the corner are discussed. In other words, a situation as in Theorem \ref{thm:2} is considered around $(\pm a,\pm a_2)$, and in the orginal articles (, especially, \cite{kobayashiyamamoto},) deformations and visualizations of the local functions and maps are discussed via explicit calculations and geometric observations.
\begin{Thm}
\label{thm:2}
In Theorem \ref{thm:1}, consider the case $a_{-1,1}=a_{1,1}=0$ and $a_2:=-a_{-1,2}=a_{1,2}>0$. In this case, the statement {\rm (}\ref{thm:1.3}{\rm)} can be replaced by $X_{D_{S_a,a_{-1,1}},m_1,m_2}:=\{(x_1,x_2,{(y_{1,j_1})}_{j_1=1}^{m_1},{(y_{2,j_2})}_{j_2=1}^{m_2}) \mid a^2-{x_1}^2-{\Sigma}_{j_1}^{m_1} {y_{j_1}}^2=0, {a_{2}}^2-{x_2}^2-{\Sigma}_{j_2}^{m_2} {y_{j_2}}^2=0\} \subset {\mathbb{R}}^{m+2}$ and the corresponding remaining statements also hold.
\end{Thm} 
\begin{proof}
Here, we explain our proof based on that of Theorem \ref{thm:1} shortly. In short, in each of Cases 1-1, 1-2, 1-3, and 1-4, we omit "$t$" and we can prove in a same way.   
\end{proof}

\subsection{Theorem \ref{thm:0} revisited as our main result.}
\label{subsec:3.4}
We present Theorem \ref{thm:0} in revised and rigorous ways. Presented statements are partially presented in \cite{kitazawa8}.

Theorem \ref{thm:3} is partially presented in \cite{kitazawa8} (\cite[Theorem 3]{kitazawa8}). Some are an answer to \cite[Remark 2]{kitazawa8}. The Morse-Bott function case is studied as our new study and some of the statement (\ref{thm:3.1}) is of our related new result and the statement  (\ref{thm:3.2}) is of this, in addition. On the other hands, in \cite{kitazawa8}, the region $D_{S_a,a_{-1},a_{1}}$ is a specific case in a certain class of essentially more general cases.
\begin{Thm}
\label{thm:3}
\begin{enumerate}
\item \label{thm:3.1}
 In Theorem \ref{thm:1}, ${{\pi}_{m+3,1} {\mid}_{X_{D_{S_a,a_{-1},a_{1}},m_1,m_2}}}^{\prime,{\rm P}}$ is a smooth function enjoying either the following three.
\begin{enumerate}
\item \label{thm:3.1.1}
This is the case $a_1:=a_{-1,1}=a_{1,1}$. 
The function ${{\pi}_{m+3,1} {\mid}_{X_{D_{S_a,a_{-1},a_{1}},m_1,m_2}}}^{\prime,{\rm P}}$ is a Morse-Bott function.
The set $S({{\pi}_{m+3,1} {\mid}_{X_{D_{S_a,a_{-1},a_{1}},m_1,m_2}}}^{\prime,{\rm P}})$ is the disjoint union of the following sets.
\begin{itemize}
\item The preimage ${{\pi}_{m+3,1} {\mid}_{X_{D_{S_a,a_{-1},a_{1}},m_1,m_2}}}^{-1}(0)$, which is diffeomorphic to $S^{m_1-1} \times S^{m_2}$, and the index of each connected component of this for the function  for the function ${{\pi}_{m+3,1} {\mid}_{X_{D_{S_a,a_{-1},a_{1}},m_1,m_2}}}^{\prime,{\rm P}}$ is $1$. The values of the function ${{\pi}_{m+3,1} {\mid}_{X_{D_{S_a,a_{-1},a_{1}},m_1,m_2}}}^{\prime,{\rm P}}$ there are $\frac{1}{1+{a_1}^2}$.
\item The preimage ${{\pi}_{m+3,1} {\mid}_{X_{D_{S_a,a_{-1},a_{1}},m_1,m_2}}}^{-1}(-a)$, which is diffeomorphic to $S^{m_2}$, and the index of this for the function ${{\pi}_{m+3,1} {\mid}_{X_{D_{S_a,a_{-1},a_{1}},m_1,m_2}}}^{\prime,{\rm P}}$ is $0$. The values of the function ${{\pi}_{m+3,1} {\mid}_{X_{D_{S_a,a_{-1},a_{1}},m_1,m_2}}}^{\prime,{\rm P}}$ there are $0$.
\item The preimage ${{\pi}_{m+3,1} {\mid}_{X_{D_{S_a,a_{-1},a_{1}},m_1,m_2}}}^{-1}(a)$, which is diffeomorphic to $S^{m_2}$, and the index of this for ${{\pi}_{m+3,1} {\mid}_{X_{D_{S_a,a_{-1},a_{1}},m_1,m_2}}}^{\prime,{\rm P}}$ is $0$. The values of the function ${{\pi}_{m+3,1} {\mid}_{X_{D_{S_a,a_{-1},a_{1}},m_1,m_2}}}^{\prime,{\rm P}}$ there are $0$.
\end{itemize}
\item \label{thm:3.1.2}
This is the case $a_{-1,1}a_{1,1}>0$ with $a_{-1,1} \neq a_{1,1}$. The function ${{\pi}_{m+3,1} {\mid}_{X_{D_{S_a,a_{-1},a_{1}},m_1,m_2}}}^{\prime,{\rm P}}$ is, around some connected components, Morse-Bott functions. The set $S({{\pi}_{m+3,1} {\mid}_{X_{D_{S_a,a_{-1},a_{1}},m_1,m_2}}}^{\prime,{\rm P}})$ is the disjoint union of the following sets. 
\begin{itemize}
\item The preimage ${{\pi}_{m+3,2} {\mid}_{X_{D_{S_a,a_{-1},a_{1}},m_1,m_2}}}^{-1}((0,a_{-1,2}))$, which is diffeomorphic to $S^{m_1-1}$. If the restriction of ${{\pi}_{m+3,1} {\mid}_{X_{D_{S_a,a_{-1},a_{1}},m_1,m_2}}}^{\prime,{\rm P}}$ to some small open neighborhood of the set is a Morse-Bott function, then the index of each connected component of which for this is $1$. 
The values of the function ${{\pi}_{m+3,1} {\mid}_{X_{D_{S_a,a_{-1},a_{1}},m_1,m_2}}}^{\prime,{\rm P}}$ there are $\frac{1}{1+{a_{-1,1}}^2}$.
\item  The preimage ${{\pi}_{m+3,2} {\mid}_{X_{D_{S_a,a_{-1},a_{1}},m_1,m_2}}}^{-1}((0,a_{1,2}))$, which is diffeomorphic to $S^{m_1-1}$.  If the restriction of ${{\pi}_{m+3,1} {\mid}_{X_{D_{S_a,a_{-1},a_{1}},m_1,m_2}}}^{\prime,{\rm P}}$ to some small open neighborhood of the set is a Morse-Bott function, then the index of each connected component of which for this is $1$.
The values of the function ${{\pi}_{m+3,1} {\mid}_{X_{D_{S_a,a_{-1},a_{1}},m_1,m_2}}}^{\prime,{\rm P}}$ there are  $\frac{1}{1+{a_{1,1}}^2}$.
\item The preimage ${{\pi}_{m+3,1} {\mid}_{X_{D_{S_a,a_{-1},a_{1}},m_1,m_2}}}^{-1}(-a)$, which is diffeomorphic to $S^{m_2}$. The restriction of ${{\pi}_{m+3,1} {\mid}_{X_{D_{S_a,a_{-1},a_{1}},m_1,m_2}}}^{\prime,{\rm P}}$ to some small open neighborhood of the set is a Morse-Bott function and the index of the subset for this is $0$. The values of the function ${{\pi}_{m+3,1} {\mid}_{X_{D_{S_a,a_{-1},a_{1}},m_1,m_2}}}^{\prime,{\rm P}}$ there are $0$.
\item The preimage ${{\pi}_{m+3,1} {\mid}_{X_{D_{S_a,a_{-1},a_{1}},m_1,m_2}}}^{-1}(a)$, which is diffeomorphic to $S^{m_2}$. The restriction of ${{\pi}_{m+3,1} {\mid}_{X_{D_{S_a,a_{-1},a_{1}},m_1,m_2}}}^{\prime,{\rm P}}$ to some small open neighborhood of the set is a Morse-Bott function and the index of the subset for this is $0$. The values of the function ${{\pi}_{m+3,1} {\mid}_{X_{D_{S_a,a_{-1},a_{1}},m_1,m_2}}}^{\prime,{\rm P}}$ there are $0$.
\end{itemize}
\item \label{thm:3.1.3}
This is the case $a_{-1,1}a_{1,1}<0$.  The set $S({{\pi}_{m+3,1} {\mid}_{X_{D_{S_a,a_{-1},a_{1}},m_1,m_2}}}^{\prime,{\rm P}})$ is the disjoint union of the four sets of the case {\rm (}\ref{thm:3.1.2}{\rm )} and the set as follows.
We can define the unique number $t_0$ with $-1 \leq t_0 \leq 1$  and $a_{-1,1}(-\frac{1}{2}t_0+\frac{1}{2})+a_{1,1}(\frac{1}{2}t_0+\frac{1}{2})=0$, and the set is, the preimage ${{\pi}_{m+3,2} {\mid}_{X_{D_{S_a,a_{-1},a_{1}},m_1,m_2}}}^{-1}((0,a_{-1,2}(-\frac{1}{2}t_0+\frac{-1}{2})+a_{1,2}(\frac{1}{2}t_0+\frac{1}{2})))$, which is diffeomorphic to $S^{m_1-1} \times S^{m_2-1}$. If the restriction of ${{\pi}_{m+3,1} {\mid}_{X_{D_{S_a,a_{-1},a_{1}},m_1,m_2}}}^{\prime,{\rm P}}$ to some small open neighborhood of the set is a Morse-Bott function, then the index of each connected component of this for ${{\pi}_{m+3,1} {\mid}_{X_{D_{S_a,a_{-1},a_{1}},m_1,m_2}}}^{\prime,{\rm P}}$ is $2$. The values of the function ${{\pi}_{m+3,1} {\mid}_{X_{D_{S_a,a_{-1},a_{1}},m_1,m_2}}}^{\prime,{\rm P}}$ there are  $1$.
\end{enumerate}
\item  \label{thm:3.2}
In addition, a case of Morse-Bott function ${{\pi}_{m+3,1} {\mid}_{X_{D_{S_a,a_{-1},a_{1}},m_1,m_2}}}^{\prime,{\rm P}}$ in {\rm (}\ref{thm:3.1.3}{\rm )} can be obtained by the condition $a_{1,1}=-a_{-1,1}$ with, for example. 
\item \label{thm:3.3}
In Theorem \ref{thm:2}, we have a corresponding fact and the statements are same as in the case of Theorem \ref{thm:1} with the case {\rm (}\ref{thm:3.1.1}{\rm )}, here.
\end{enumerate}
\end{Thm}
\begin{proof}
Except additional studies on the case of a Morse-Bott function ${{\pi}_{m+3,1} {\mid}_{X_{D_{S_a,a_{-1},a_{1}},m_1,m_2}}}^{\prime,{\rm P}}$, we have proven statements here in \cite{kitazawa8}. Of course we do not need related non-trivial facts.

We first prove the statement (\ref{thm:3.1}).



By Proposition \ref{prop:2} and respecting affine isomorphisms, the structures of ${\pi}_{m+3,1} {\mid}_{X_{D_{S_a,a_{-1},a_{1}},m_1,m_2}}$, $S( {\pi}_{m+3,1} {\mid}_{X_{D_{S_a,a_{-1},a_{1}},m_1,m_2}})$, the local change of the function ${{\pi}_{m+3,1} {\mid}_{X_{D_{S_a,a_{-1},a_{1}},m_1,m_2}}}^{\prime,{\rm P}}$ according to local change of $x_1$ and $t$, and so on,  
the critical set $S({{\pi}_{m+3,1} {\mid}_{X_{D_{S_a,a_{-1},a_{1}},m_1,m_2}}}^{\prime,{\rm P}})$ of the function  ${{\pi}_{m+3,1} {\mid}_{X_{D_{S_a,a_{-1},a_{1}},m_1,m_2}}}^{\prime,{\rm P}}$ must be a subset of

 ${{\pi}_{m+3,2} {\mid}_{X_{D_{S_a,a_{-1},a_{1}},m_1,m_2}}}^{-1}(\{(0,x_2) \mid x_2 \in \mathbb{R}\} \sqcup \{(-a,x_2) \mid x_2 \in \mathbb{R}\} \sqcup \{(a,x_2) \mid x_2 \in \mathbb{R}\})$.

We investigate $S({{\pi}_{m+3,1} {\mid}_{X_{D_{S_a,a_{-1},a_{1}},m_1,m_2}}}^{\prime,{\rm P}})$ and calculate the values of ${{\pi}_{m+3,1} {\mid}_{X_{D_{S_a,a_{-1},a_{1}},m_1,m_2}}}^{\prime,{\rm P}}$ on each connected component of $S({{\pi}_{m+3,1} {\mid}_{X_{D_{S_a,a_{-1},a_{1}},m_1,m_2}}}^{\prime,{\rm P}})$.

The values of ${{\pi}_{m+3,1} {\mid}_{X_{D_{S_a,a_{-1},a_{1}},m_1,m_2}}}^{\prime,{\rm P}}$ on ${{\pi}_{m+3,2} {\mid}_{X_{D_{S_a,a_{-1},a_{1}},m_1,m_2}}}^{-1}(\{(-a,x_2) \mid x_2 \in \mathbb{R}\} \sqcup \{(a,x_2) \mid x_2 \in \mathbb{R}\})$ are $0$ and ${{\pi}_{m+3,2} {\mid}_{X_{D_{S_a,a_{-1},a_{1}},m_1,m_2}}}^{-1}(\{(-a,x_2) \mid x_2 \in \mathbb{R}\} \sqcup \{(a,x_2) \mid x_2 \in \mathbb{R}\})$ is the level set ${{{\pi}_{m+3,1} {\mid}_{X_{D_{S_a,a_{-1},a_{1}},m_1,m_2}}}^{\prime,{\rm P}}}^{-1}(0)$ of the function ${{\pi}_{m+3,1} {\mid}_{X_{D_{S_a,a_{-1},a_{1}},m_1,m_2}}}^{\prime,{\rm P}}$. This set is also a smooth submanifold of $X_{D_{S_a,a_{-1},a_{1}},m_1,m_2}$ and diffeomorphic to $S^{m_2} \sqcup S^{m_2}$. The index of each connected component of the set for the function ${{\pi}_{m+3,1} {\mid}_{X_{D_{S_a,a_{-1},a_{1}},m_1,m_2}}}^{\prime,{\rm P}}$ is $0$. Note that this level set is for all three cases (\ref{thm:3.1.1}), (\ref{thm:3.1.2}), and (\ref{thm:3.1.3}).

In the case (\ref{thm:3.1.1}), the preimage ${{\pi}_{m+3,2} {\mid}_{X_{D_{S_a,a_{-1},a_{1}},m_1,m_2}}}^{-1}(\{(0,x_2) \mid x_2 \in \mathbb{R}\})$ is also a subset of the set $S({{\pi}_{m+3,1} {\mid}_{X_{D_{S_a,a_{-1},a_{1}},m_1,m_2}}}^{\prime,{\rm P}})$ and this is also a smooth submanifold of $X_{D_{S_a,a_{-1},a_{1}},m_1,m_2}$ and diffeomorphic to $S^{m_1-1} \times S^{m_2}$. The values of ${{\pi}_{m+3,1} {\mid}_{X_{D_{S_a,a_{-1},a_{1}},m_1,m_2}}}^{\prime,{\rm P}}$ there and ${{\pi}_{m+3,2} {\mid}_{X_{D_{S_a,a_{-1},a_{1}},m_1,m_2}}}^{-1}(\{(0,x_2) \mid x_2 \in \mathbb{R}\})$ is the level set ${{{\pi}_{m+3,1} {\mid}_{X_{D_{S_a,a_{-1},a_{1}},m_1,m_2}}}^{\prime,{\rm P}}}^{-1}(1)$ of the function. The index of each connected component of the set for the function ${{\pi}_{m+3,1} {\mid}_{X_{D_{S_a,a_{-1},a_{1}},m_1,m_2}}}^{\prime,{\rm P}}$  is $1$. 

In the cases (\ref{thm:3.1.2}) and (\ref{thm:3.1.3}), $S({{\pi}_{m+3,1} {\mid}_{X_{D_{S_a,a_{-1},a_{1}},m_1,m_2}}}^{\prime,{\rm P}})$ must be subsets of

${{\pi}_{m+3,2} {\mid}_{X_{D_{S_a,a_{-1},a_{1}},m_1,m_2}}}^{-1}(\{(0,x_2) \mid x_2 \in \{a_{-1,2},0,a_{1,2}\}\} \sqcup \{(-a,x_2) \mid x_2 \in \mathbb{R}\} \sqcup \{(a,x_2) \mid x_2 \in \mathbb{R}\})$.


The preimages ${{\pi}_{m+3,2} {\mid}_{X_{D_{S_a,a_{-1},a_{1}},m_1,m_2}}}^{-1}((0,a_{-1,2}))$ and ${{\pi}_{m+3,2} {\mid}_{X_{D_{S_a,a_{-1},a_{1}},m_1,m_2}}}^{-1}((0,a_{1,2}))$ are smooth submanifolds in $X_{D_{S_a,a_{-1},a_{1}},m_1,m_2}$ and diffeomorphic to $S^{m_1-1}$. They are also unions of connected components of $S({{\pi}_{m+3,1} {\mid}_{X_{D_{S_a,a_{-1},a_{1}},m_1,m_2}}}^{\prime,{\rm P}})$.

The tangent vector $v_{p,{\pi}_{m+3,1} {\mid}_{X_{D_{S_a,a_{-1},a_{1}},m_1,m_2}},+}$ at $p \in X_{D_{S_a,a_{-1},a_{1}},m_1,m_2}$ as in the Subsection \ref{subsec:3.1} is parallel to a vector whose $j$-th components are $0$ for $j \geq 3$ if and only if $x_1=0$ in $p=(x_1,x_2,t,{(y_{1,j_1})}_{j_1=1}^{m_1},{(y_{2,j_2})}_{j_2=1}^{m_2})$.

From this observation, the values of ${{\pi}_{m+3,1} {\mid}_{X_{D_{S_a,a_{-1},a_{1}},m_1,m_2}}}^{\prime,{\rm P}}$ on ${{\pi}_{m+3,2} {\mid}_{X_{D_{S_a,a_{-1},a_{1}},m_1,m_2}}}^{-1}((0,a_{-1,2}))$ are calculated to be $\frac{1}{{a_{-1,1}}^2+1} $ and the values of ${{\pi}_{m+3,1} {\mid}_{X_{D_{S_a,a_{-1},a_{1}},m_1,m_2}}}^{\prime,{\rm P}}$ on ${{\pi}_{m+3,2} {\mid}_{X_{D_{S_a,a_{-1},a_{1}},m_1,m_2}}}^{-1}((0,a_{1,2}))$ are calculated to be $\frac{1}{{a_{1,1}}^2+1} $. 

In the cases (\ref{thm:3.1.2}) and (\ref{thm:3.1.3}), for the index of each connected component of these sets ${{\pi}_{m+3,2} {\mid}_{X_{D_{S_a,a_{-1},a_{1}},m_1,m_2}}}^{-1}((0,a_{\pm 1,2}))$ for the function ${{\pi}_{m+3,1} {\mid}_{X_{D_{S_a,a_{-1},a_{1}},m_1,m_2}}}^{\prime,{\rm P}}$ under the condition that this function is Morse-Bott around its small neighborhood, we investigate several preimages of subsets in ${\mathbb{R}}^2$ by the map ${\pi}_{m+3,2} {\mid}_{X_{D_{S_a,a_{-1},a_{1}},m_1,m_2}}$.

First we consider $\{(x_1,\pm a_{\pm 1,2}) \mid x_1 \in \mathbb{R}\}$ and the preimage ${{\pi}_{m+3,2} {\mid}_{X_{D_{S_a,a_{-1},a_{1}},m_1,m_2}}}^{-1}(\{(x_1,\pm a_{\pm 1,2}) \mid x_1 \in \mathbb{R}\})$. We can have a case as in Proposition \ref{prop:2} on the $m_1$-dimensional unit sphere, where "$m:=m_1$" there. Second we consider $\{(0,x_2) \mid x_2 \in \mathbb{R}\}$ and the preimage ${{\pi}_{m+3,2} {\mid}_{X_{D_{S_a,a_{-1},a_{1}},m_1,m_2}}}^{-1}(\{(0,x_2) \mid x_2 \in \mathbb{R}\})$. We can have a case as in Proposition \ref{prop:2} on the $m_2$-dimensional unit sphere, where "$m:=m_2$" (in Proposition \ref{prop:2}). The $m_1$-dimensional sphere and the $m_2$-dimensional sphere here intersect in a one-point set and the sum of these two corresponding tangent vectors at the point is the tangent vector of $X_{D_{S_a,a_{-1},a_{1}},m_1,m_2}$ there. Remember the tangent vector $v_{p,{\pi}_{m+3,1} {\mid}_{X_{D_{S_a,a_{-1},a_{1}},m_1,m_2}},+}$ at $p=(x_1,x_2,t,{(y_{1,j_1})}_{j_1=1}^{m_1},{(y_{2,j_2})}_{j_2=1}^{m_2}) \in X_{D_{S_a,a_{-1},a_{1}},m_1,m_2}$ with $x_1=0$. From this, at $p=(0,(-\frac{1}{2}t+\frac{1}{2})(a_{-1,1}x_1+a_{-1,2})+(\frac{1}{2}t+\frac{1}{2})(a_{1,1}x_1+a_{1,2}),t,{(y_{1,j_1})}_{j_1=1}^{m_1},{(y_{2,j_2})}_{j_2=1}^{m_2})=(0,a_{-1,2}(-\frac{1}{2}t+\frac{1}{2})+a_{1,2}(\frac{1}{2}t+\frac{1}{2}),t,{(y_{1,j_1})}_{j_1=1}^{m_1},{(y_{2,j_2})}_{j_2=1}^{m_2}) \in X_{D_{S_a,a_{-1},a_{1}},m_1,m_2}$, $v_{p,{\pi}_{m+3,1} {\mid}_{X_{D_{S_a,a_{-1},a_{1}},m_1,m_2}},+}$, $v_{p,{\pi}_{m+3,1} {\mid}_{X_{D_{S_a,a_{-1},a_{1}},m_1,m_2}},+}$ is mapped to a tangent vector at ${\pi}_{m+3,1}(p)$ of ${\mathbb{R}}$. By our construction, the resulting vector is understood to be tangent to the straight line $\{(x_1,x_2) \mid x_2=-\frac{1}{2}t+\frac{1}{2})(a_{-1,1}x_1+a_{-1,2})+(\frac{1}{2}t+\frac{1}{2})(a_{1,1}x_1+a_{1,2})\}$. The square of its length is $\frac{1}{{\{a_{-1,1}(-\frac{1}{2}t+\frac{1}{2})+a_{1,1}(\frac{1}{2}t+\frac{1}{2})\}}^2+1}$, by the differential of ${\pi}_{m+3,1}$. The observation on indices of each connected component of these sets ${{\pi}_{m+3,2} {\mid}_{X_{D_{S_a,a_{-1},a_{1}},m_1,m_2}}}^{-1}((0,a_{\pm 1,2}))$ for the function ${{\pi}_{m+3,1} {\mid}_{X_{D_{S_a,a_{-1},a_{1}},m_1,m_2}}}^{\prime,{\rm P}}$ are obtained by local structures of the functions and the maps, under the condition that these functions are Morse-Bott functions around these sets.


In the cases (\ref{thm:3.1.1}) and (\ref{thm:3.1.2}), the sets $S({{\pi}_{m+3,1} {\mid}_{X_{D_{S_a,a_{-1},a_{1}},m_1,m_2}}}^{\prime,{\rm P}})$ are all investigated.
In the case (\ref{thm:3.1.3}), we also need ${{\pi}_{m+3,2} {\mid}_{X_{D_{S_a,a_{-1},a_{1}},m_1,m_2}}}^{-1}((0,0))$ to complete the set $S({{\pi}_{m+3,1} {\mid}_{X_{D_{S_a,a_{-1},a_{1}},m_1,m_2}}}^{\prime,{\rm P}})$. This subset is a smooth submanifold of $X_{D_{S_a,a_{-1},a_{1}},m_1,m_2}$ and diffeomorphic to $S^{m_1-1} \times S^{m_2-1}$. The value of ${{\pi}_{m+3,1} {\mid}_{X_{D_{S_a,a_{-1},a_{1}},m_1,m_2}}}^{\prime,{\rm P}}$ there is $1$ and the set is also the level set ${{{\pi}_{m+3,1} {\mid}_{X_{D_{S_a,a_{-1},a_{1}},m_1,m_2}}}^{\prime,{\rm P}}}^{-1}(1)$ of the function. If around the set ${{\pi}_{m+3,1} {\mid}_{X_{D_{S_a,a_{-1},a_{1}},m_1,m_2}}}^{\prime,{\rm P}}$ is Morse-Bott, then the index of each connected component of the set for the function ${{\pi}_{m+3,1} {\mid}_{X_{D_{S_a,a_{-1},a_{1}},m_1,m_2}}}^{\prime,{\rm P}}$ is calculated to be $2$.

We prove the statement (\ref{thm:3.2}). Remember the tangent vector $v_{p,{\pi}_{m+3,1} {\mid}_{X_{D_{S_a,a_{-1},a_{1}},m_1,m_2}},+}$ at $p=(x_1,x_2,t,{(y_{1,j_1})}_{j_1=1}^{m_1},{(y_{2,j_2})}_{j_2=1}^{m_2}) \in X_{D_{S_a,a_{-1},a_{1}},m_1,m_2}$ with $x_1=0$. 
Remember that at $p=(0,a_{-1,2}(-\frac{1}{2}t+\frac{1}{2})+a_{1,2}(\frac{1}{2}t+\frac{1}{2}),t,{(y_{1,j_1})}_{j_1=1}^{m_1},{(y_{2,j_2})}_{j_2=1}^{m_2}) \in X_{D_{S_a,a_{-1},a_{1}},m_1,m_2}$, the tangent vector $v_{p,{\pi}_{m+3,1} {\mid}_{X_{D_{S_a,a_{-1},a_{1}},m_1,m_2}},+}$ is mapped to a tangent vector at ${\pi}_{m+3,1}(p)$ of ${\mathbb{R}}$ the square of whose  length is $\frac{1}{{\{a_{-1,1}(-\frac{1}{2}t+\frac{1}{2})+a_{1,1}(\frac{1}{2}t+\frac{1}{2})\}}^2+1}$, by the differential of ${\pi}_{m+3,1}$.
We consider $a_{1,1}=-a_{-1,1}$. the square of whose length of the previous vector is calculated as follows.

$\frac{1}{{\{a_{-1,1}(-\frac{1}{2}t+\frac{1}{2})+a_{1,1}(\frac{1}{2}t+\frac{1}{2})\}}^2+1} = \frac{1}{{(a_{1,1}t)}^2+1}= \frac{1}{{(a_{1,1})}^2-{(a_{1,1})}^2 {\Sigma}_{j_2=1}^{m_2} ({y_{2,j_2}})^2+1}$.

The value of its 1st partial derivative by $y_{2,j_{2,0}}$ at ${(y_{2,j_2})}_{j_2=1}^{m_2}$ 
is calculated to be \\
$$
-\frac{2y_{2,j_2}{a_{1,1}}^2}{{{((a_{1,1})}^2-{(a_{1,1})}^2 {\Sigma}_{j_2=1}^{m_2} ({y_{2,j_2}})^2+1)}^2}
$$
and that of the 2nd partial derivative is calculated to be
$$\frac{-2{a_{1,1}}^2({(a_{1,1})}^2-{(a_{1,1})}^2 {\Sigma}_{j_2=1}^{m_2} ({y_{2,j_2}})^2+1)-8{(a_{1,1})}^4{y_{2,j_{2,0}}}^2}{{{((a_{1,1})}^2-{(a_{1,1})}^2 {\Sigma}_{j_2=1}^{m_2} ({y_{2,j_2}})^2+1)}^3}<0$$
and this implies that ${{\pi}_{m+3,1} {\mid}_{X_{D_{S_a,a_{-1},a_{1}},m_1,m_2}}}^{\prime,{\rm P}}$ is a Morse-Bott function around every connected component of its critical set. This completes the proof of the statement (\ref{thm:3.2}).

By the structures of the manifolds and the maps, we can have the corresponding fact for Theorem \ref{thm:2}, easily. We have shown the statement (\ref{thm:3.3}).

This completes the proof.
\end{proof}
\begin{Rem}
\label{rem:1}
Related to Theorem \ref{thm:3} (\ref{thm:3.2}) with (\ref{thm:3.1.2}) and (\ref{thm:3.1.3}), we believe that we can have Morse-Bott functions in generic cases, in other general situations. However, related calculations seem to be more complicated and we do not know conditions for the functions ${{\pi}_{m+3,1} {\mid}_{X_{D_{S_a,a_{-1},a_{1}},m_1,m_2}}}^{\prime,{\rm P}}$ to be Morse-Bott, explicitly.
\end{Rem}
Theorem \ref{thm:4} is also of our new result. This is on topological structures and combinatorial ones of the functions ${{\pi}_{m+3,1} {\mid}_{X_{D_{S_a,a_{-1},a_{1}},m_1,m_2}}}^{\prime,{\rm P}}$ and their Reeb digraphs and level sets are studied.
\begin{Thm}
\label{thm:4}
For Theorem \ref{thm:3}, the following hold, in addition.

\begin{enumerate}
\item \label{thm:4.1} The Reeb digraph $\overrightarrow{R_{{{\pi}_{m+3,1} {\mid}_{X_{D_{S_a,a_{-1},a_{1}},m_1,m_2}}}^{\prime,{\rm P}}}}$ of ${{\pi}_{m+3,1} {\mid}_{X_{D_{S_a,a_{-1},a_{1}},m_1,m_2}}}^{\prime,{\rm P}}$  is as in Proposition \ref{prop:2} in the case of Theorem \ref{thm:3} {\rm (}\ref{thm:3.1.1}{\rm )}. In the case $m_1 \geq 2$ {\rm (}$m_1=1${\rm )}, this is as in the left {\rm (}resp. right{\rm )} figure in Figure \ref{fig:1}. Each contour of ${{\pi}_{m+3,1} {\mid}_{X_{D_{S_a,a_{-1},a_{1}},m_1,m_2}}}^{\prime,{\rm P}}$ is represented as a connected component of  ${{\pi}_{m+3,2} {\mid}_{X_{D_{S_a,a_{-1},a_{1}},m_1,m_2}}}^{-1}(\{(t,x_2) \mid x_2 \in \mathbb{R}\})$ for a fixed number $-a \leq t \leq a$ and if it is not critical, then it is diffeomorphic to $S^{m_1-1} \times S^{m_2}$ {\rm (}resp. $S^{m_2}${\rm )}. 
\item \label{thm:4.2}

The Reeb digraph $\overrightarrow{R_{{{\pi}_{m+3,1} {\mid}_{X_{D_{S_a,a_{-1},a_{1}},m_1,m_2}}}^{\prime,{\rm P}}}}$ of ${{\pi}_{m+3,1} {\mid}_{X_{D_{S_a,a_{-1},a_{1}},m_1,m_2}}}^{\prime,{\rm P}}$  is as follows, in the case of Theorem \ref{thm:3} {\rm (}\ref{thm:3.1.2}{\rm )}.
\begin{itemize}
\item This is for the case $m_1=2$. The digraph is a digraph with exactly four vertices $v_{0,1}$, $v_{0,2}$, $v_a$ and $v_A$ and exactly three edges two of which depart from $v_{0,i}$ {\rm (}$i=1,2${\rm )} and enter $v_a$ and one of which departs from $v_a$ and enters $v_A$.   
\item This is for the case $m_1=1$. The digraph is a digraph with exactly six vertices $v_{0,i}$, $v_{a,i}$, and $v_{A,i}$ {\rm (}$i=1,2${\rm )} and exactly six edges four of which depart from $v_{0,i_1}$ and enter $v_{a,i_2}$ {\rm (}$(i_1,i_2) \in \{1,2\} \times \{1,2\}${\rm )} and two of which depart from $v_{a,i}$ and enter $v_{A,i}$ {\rm (}$i=1,2${\rm )} .  
\end{itemize}

Hereafter,the minimum in the set $\{\frac{1}{1+{a_{-1,1}}^2}, \frac{1}{1+{a_{1,1}}^2}\}$ is denoted by $m_{a}$ and the maximum there is denoted by $m_A$. It may hold that $m_a=m_A$ or $m_A=1$ in this case,  {\rm (}\ref{thm:4.2}{\rm )}.

In this case, the level set ${{{\pi}_{m+3,1} {\mid}_{X_{D_{S_a,a_{-1},a_{1}},m_1,m_2}}}^{\prime,{\rm P}}}^{-1}(u)$ {\rm (}$0<u \leq m_A${\rm )} of the function is as follows.

Hereafter, in this theorem, connected and compact curves $C_{{\rm s},1}$, $C_{{\rm s},2}$, and $C_{\rm s}$ which are smooth densely, are used. 
\begin{itemize}
\item {\rm (}The case $0<u<m_a$.{\rm )} The preimage ${{\pi}_{m+3,2} {\mid}_{X_{D_{S_a,a_{-1},a_{1}},m_1,m_2}}}^{-1}(C_{{\rm s},1} \sqcup C_{{\rm s},2})$ 
of the disjoint union of two connected and compact curves $C_{{\rm s},1}$ and $C_{{\rm s},2}$ of ${\overline{D_{S_a,a_{-1},a_{1}}}}^{{\mathbb{R}}^2}$ homeomorphic to $D^1$. In the case $m_1 \neq 1$, this is homeomorphic to the disjoint union $(S^{m_1-1} \times S^{m_2}) \sqcup (S^{m_1-1} \times S^{m_2})$ and diffeomorphic to it in the case $m \neq 5,6$ or $(m_1,m_2)=(3,3), (4,2)$. In the case $m_1=1$, this is homeomorohic to the disjoint union $S^{m_2} \sqcup S^{m_2} \sqcup S^{m_2} \sqcup S^{m_2}$ and diffeomorphic to it in the case $m_2 \neq 4$. 
\item {\rm (}The case $m_a<u<m_A$.{\rm )} The preimage ${{\pi}_{m+3,2} {\mid}_{X_{D_{S_a,a_{-1},a_{1}},m_1,m_2}}}^{-1}(C_{\rm s})$ of some connected and compact curve $C_{\rm s}$ of $\overline{D_{S_a,a_{-1},a_{1}}}^{{\mathbb{R}}^2}$ homeomorphic to $D^1$.
In the case $m_1 \neq 1$, this is homeomorphic to $S^{m_1-1} \times S^{m_2}$ and diffeomorphic to it in the case $m \neq 5,6$ or $(m_1,m_2)=(3,3), (4,2)$. In the case $m_1=1$, this is homeomorphic to the disjoint union $S^{m_2} \sqcup S^{m_2}$ and diffeomorphic to it in the case $m_2 \neq 4$. 
\item {\rm (}The case $u=m_a, m_A$.{\rm )}
The preimage ${{\pi}_{m+3,2} {\mid}_{X_{D_{S_a,a_{-1},a_{1}},m_1,m_2}}}^{-1}(C_{\rm s})$ of some connected and compact curve $C_{\rm s}$ of $\overline{D_{S_a,a_{-1},a_{1}}}^{{\mathbb{R}}^2}$ homeomorphic to $S^1$ in the case $m_a=m_A$ and $D^1$ in the case $m_a \neq m_A$.

\end{itemize}
\item \label{thm:4.3} The Reeb digraph $\overrightarrow{R_{{{\pi}_{m+3,1} {\mid}_{X_{D_{S_a,a_{-1},a_{1}},m_1,m_2}}}^{\prime,{\rm P}}}}$ of ${{\pi}_{m+3,1} {\mid}_{X_{D_{S_a,a_{-1},a_{1}},m_1,m_2}}}^{\prime,{\rm P}}$  is as follows, in the case of Theorem \ref{thm:3} {\rm (}\ref{thm:3.1.3}{\rm )}.
\begin{itemize}
\item This is for the case $m_1 \geq 2$ and $m_2 \geq 2$. The digraph is a digraph with exactly five vertices $v_{0,1}$, $v_{0,2}$, $v_a$, $v_A$, and $v_1$ and exactly four edges two of which depart from $v_{0,i}$ {\rm (}$i=1,2${\rm )} and enter $v_a$, one of which departs from $v_a$ and enters $v_A$, and the remaining one of which departs from $v_A$ and enters $v_1$.   
\item This is for the case $m_1=1$ and $m_2 \geq 2$. The digraph is a digraph with exactly eight vertices $v_{0,i}$, $v_{a,i}$, $v_{A,i}$, and $v_{1,i}$ {\rm (}$i=1,2${\rm )} and exactly eight edges four of which depart from $v_{0,i_1}$ and enter $v_{a,i_2}$ {\rm (}$(i_1,i_2) \in \{1,2\} \times \{1,2\}${\rm )}, two of which depart from $v_{a,i}$ and enter $v_{A,i}$ {\rm (}$i=1,2${\rm )}, and the remaining two of which depart from $v_{A,i}$ and enter $v_{1,i}$ {\rm (}$i=1,2${\rm )}.
\item This is for the case $m_1 \geq 2$ and $m_2=1$. The digraph is a digraph with exactly six vertices $v_{0,i}$, $v_{a}$, $v_{A}$, and $v_{1,i}$ {\rm (}$i=1,2${\rm )} and exactly five edges two of which depart from $v_{0,i}$ {\rm (}$i=1,2${\rm )} and enter $v_a$, one of which departs from $v_a$ and enters $v_A$, and the remaining two of which depart from $v_A$ and enter  $v_{1,i}$ {\rm (}$i=1,2${\rm )} .   
\item This is for the case $(m_1,m_2)=(1,1)$. The digraph is a digraph with exactly eight vertices $v_{0,i}$, $v_{a,i}$, $v_{A,i}$, and $v_{1,i^{\prime}}$ {\rm (}$i=1,2$ and $i^{\prime}=1,2,3,4${\rm )} and exactly ten edges four of which depart from $v_{0,i_1}$ and enter $v_{a,i_2}$ {\rm (}$(i_1,i_2) \in \{1,2\} \times \{1,2\}${\rm )}, two of which depart from $v_{a,i}$ and enter $v_{A,i}$ {\rm (}$i=1,2${\rm )}, and the remaining four of which depart from $v_{A,i_1}$ and enter $v_{1,i_2}$ {\rm (}$(i_1,i_2) \in \{(1,1),(1,2),(2,3),(2,4)\}${\rm )}.
\end{itemize}
Furthermore, the level set ${{{\pi}_{m+3,1} {\mid}_{X_{D_{S_a,a_{-1},a_{1}},m_1,m_2}}}^{\prime,{\rm P}}}^{-1}(u)$ {\rm (}$0<u<1${\rm )} of the function is as follows.
\begin{itemize}
\item {\rm (}The case $0<u<m_a$.{\rm )}
The preimage ${{\pi}_{m+3,2} {\mid}_{X_{D_{S_a,a_{-1},a_{1}},m_1,m_2}}}^{-1}(C_{{\rm s},1} \sqcup C_{{\rm s},2})$ 
of the disjoint union of two connected and compact curves $C_{{\rm s},1}$ and $C_{{\rm s},2}$ of ${\overline{D_{S_a,a_{-1},a_{1}}}}^{{\mathbb{R}}^2}$ homeomorphic to $D^1$.

 In the case $m_1=1$, this is homeomorphic to the disjoint union $(S^{m_1-1} \times S^{m_2}) \sqcup (S^{m_1-1} \times S^{m_2})$ and diffeomorphic to it in the case $m \neq 5$  or $(m_1,m_2)=(3,3),(4,2)$. In the case $m_1=1$, this is homeomorphic to the disjoint union $S^{m_2} \sqcup S^{m_2}$ and diffeomorphic to it in the case $m_2 \neq 4$. 
\item {\rm (}The case $m_a<u<m_A$.{\rm )} 
The preimage ${{\pi}_{m+3,2} {\mid}_{X_{D_{S_a,a_{-1},a_{1}},m_1,m_2}}}^{-1}(C_{\rm s})$ of some connected and compact curve $C_{\rm s}$ of $\overline{D_{S_a,a_{-1},a_{1}}}^{{\mathbb{R}}^2}$ homeomorphic to $S^1$.
This is homeomorphic to $S^{m_1-1} \times S^{m_2}$ and diffeomorphic to it in the case $m \neq 5,6$ or $(m_1,m_2)=(3,3),(4,2)$. In the case $m_1=1$, this is homeomorphic to $S^{m_2} \sqcup S^{m_2}$ and diffeomorphic to it in the case $m_2 \neq 4$. 

\item {\rm (}The case $m_A<u<1$.{\rm )} 
The preimage ${{\pi}_{m+3,2} {\mid}_{X_{D_{S_a,a_{-1},a_{1}},m_1,m_2}}}^{-1}(C_{\rm s})$ of some connected and compact curve $C_{\rm s}$ of $\overline{D_{S_a,a_{-1},a_{1}}}^{{\mathbb{R}}^2}$ homeomorphic to $S^1$.
In the case $m_1 \neq 1$, this is homeomorphic to $S^1 \times S^{m_1-1} \times S^{m_2-1}$ and diffeomorphic to it in the case $m \neq 5,6$. In the case $m_1=1$ and $m_2 \neq 1$, this is homeomorphic to the disjoint union of two copies of $S^1 \times S^{m_2-1}$ and diffeomorphic to it in the case $m_2 \neq 4,5$. In the case $m_1 \neq 1$ and $m_2=1$, this is homeomorphic to the disjoint union of two copies of $S^1 \times S^{m_1-1}$ and diffeomorphic to it in the case $m_1 \neq 4,5$. In the case $m_1=1$ and $m_2=1$, this is homeomorphic to the disjoint union of four copies of $S^1$. 
\item {\rm (}The case $u=m_a, m_A$.{\rm )}
The preimage ${{\pi}_{m+3,2} {\mid}_{X_{D_{S_a,a_{-1},a_{1}},m_1,m_2}}}^{-1}(C_{\rm s})$ of some connected and compact curve $C_{\rm s}$ of $\overline{D_{S_a,a_{-1},a_{1}}}^{{\mathbb{R}}^2}$ homeomorphic to $S^1$.
\end{itemize}
\end{enumerate}
\end{Thm}
\begin{proof}
STEP 4-1 The case (\ref{thm:4.1}). \\
We prove the case (\ref{thm:4.1}).
By our construction, each contour of ${{\pi}_{m+3,1} {\mid}_{X_{D_{S_a,a_{-1},a_{1}},m_1,m_2}}}^{\prime,{\rm P}}$ is represented as a connected component of  ${{\pi}_{m+3,2} {\mid}_{X_{D_{S_a,a_{-1},a_{1}},m_1,m_2}}}^{-1}(\{(x_{1,0},x_2) \mid x_2 \in \mathbb{R}\})$ for a fixed number $-a \leq x_{1,0} \leq a$.  \\
\ \\
STEP 4-2 The cases  (\ref{thm:4.2}) and (\ref{thm:4.3}). \\
\ \\
We prove the cases (\ref{thm:4.2}) and (\ref{thm:4.3}). 
${{{\pi}_{m+3,1} {\mid}_{X_{D_{S_a,a_{-1},a_{1}},m_1,m_2}}}^{\prime,{\rm P}}}^{-1}(u)$ is, in the case $u=0,1$, discussed in Theorem \ref{thm:3} (\ref{thm:3.1.2}, \ref{thm:3.1.3}).
In the case $0<u<1$, this is also the preimage ${{\pi}_{m+3,2} {\mid}_{X_{D_{S_a,a_{-1},a_{1}},m_1,m_2}}}^{-1}(C_{\rm s})$ of some connected, closed and compact subset $C_{\rm s}$ of $\overline{D_{S_a,a_{-1},a_{1}}}^{{\mathbb{R}}^2}$ or
the preimage ${{\pi}_{m+3,2} {\mid}_{X_{D_{S_a,a_{-1},a_{1}},m_1,m_2}}}^{-1}(C_{{\rm s},1} \sqcup C_{{\rm s},2})$ 
of the disjoint union of two connected, closed and compact subsets $C_{{\rm s},1}$ and $C_{{\rm s},2}$ of ${\overline{D_{S_a,a_{-1},a_{1}}}}^{{\mathbb{R}}^2}$, and explained as follows.

Let the canonical projection to the 2nd component be denoted by ${\pi}_{2,1,2}:{\mathbb{R}}^2 \rightarrow \mathbb{R}$.\\
\ \\
STEP 4-2-1 The case ${{{\pi}_{m+3,1} {\mid}_{X_{D_{S_a,a_{-1},a_{1}},m_1,m_2}}}^{\prime,{\rm P}}}^{-1}(u)$ ($0<u \leq m_A$).
\begin{itemize}
\item In the case $0<u<m_{a}$, we need disjoint two connected, closed and compact subsets $C_{{\rm s},1} \subset {\overline{D_{S_a,a_{-1},a_{1}}}}^{{\mathbb{R}}^n} \bigcap \{(x_1,x_2) \mid x_1 <0, x_2 \in \mathbb{R}\}$ and $C_{{\rm s},2} \subset {\overline{D_{S_a,a_{-1},a_{1}}}}^{{\mathbb{R}}^2} \bigcap \{(x_1,x_2) \mid x_1 >0,  x_2 \in \mathbb{R}\}$ of $D_{S_a,a_{-1},a_{1}}$ such that the restrictions of the projection ${\pi}_{2,1,2}:{\mathbb{R}}^2 \rightarrow \mathbb{R}$ to $C_{{\rm s},i}$ are injective real-valued functions for $i=1,2$, and that the intersections $C_{{\rm s},i } \bigcup S_{{\rm l},\pm 1}$ are one-point sets for $i=1,2$ and consist of all points of the boundary of $C_{{\rm s},1} \sqcup C_{{\rm s},2}$.
We can see that ${{{\pi}_{m+3,1} {\mid}_{X_{D_{S_a,a_{-1},a_{1}},m_1,m_2}}}^{\prime,{\rm P}}}^{-1}(u)$ is diffeomorphic to $S^{m_1-1} \times S^{m_2}$.
\item In the case $u=m_{a}$ with $m_{a}<m_{A}$ and $m_a=\frac{1}{1+{a_{-1,1}}^2}$ ($m_a=\frac{1}{1+{a_{1,1}}^2}$), we need one connected, closed and compact subset $C_{\rm s} \subset {\overline{D_{S_a,a_{-1},a_{1}}}}^{{\mathbb{R}}^2}$ of ${\overline{D_{S_a,a_{-1},a_{1}}}}^{{\mathbb{R}}^n}$, as follows.
\begin{itemize}
\item $C_{\rm s,-}:=C_{\rm s} \bigcap \{(x_1,x_2) \mid x_1 \leq 0,  x_2 \in \mathbb{R}\}$ is a connected, closed and compact subset of ${\overline{D_{S_a,a_{-1},a_{1}}}}^{{\mathbb{R}}^n}$, and the restriction of the projection ${\pi}_{2,1,2}:{\mathbb{R}}^2 \rightarrow \mathbb{R}$ to $C_{{\rm s},-}$ is an injective real-valued function, with the intersections $C_{{\rm s},-} \bigcap S_{{\rm l},-1}$ and $C_{{\rm s},-} \bigcap S_{{\rm l},1}$ being one-point sets $\{(p_{-,1},p_{-,2})\} \subset S_{{\rm l},-1}$ ($\{(p_{-,1},p_{-,2})\} \subset S_{{\rm l},1}$) for suitable numbers $p_{-,1}<0$ and $p_{-,2}$, and $\{(0,a_{1,2})\}$ (resp. $\{(0,a_{-1,2})\}$), respectively. 
$C_{\rm s,-}$ is homeomorphic to $D^1$. In addition, $C_{{\rm s},-} \bigcap (S_{{\rm l},-1} \bigcup S_{{\rm l},1})$ consists of all points of the boundary of $C_{{\rm s},-}$. 
\item $C_{\rm s,+}:=C_{\rm s} \bigcap \{(x_1,x_2) \mid x_1 \geq 0,  x_2 \in \mathbb{R}\}$ is a connected, closed and compact subset of ${\overline{D_{S_a,a_{-1},a_{1}}}}^{{\mathbb{R}}^n}$, and the restriction of the projection ${\pi}_{2,1,2}:{\mathbb{R}}^2 \rightarrow \mathbb{R}$ to $C_{{\rm s},+}$ is an injective real-valued function, with the intersections $C_{{\rm s},+} \bigcup S_{{\rm l},-1}$ and $C_{{\rm s},+} \bigcup S_{{\rm l},1}$ being one-point sets $\{(p_{+,1},p_{+,2})\} \subset S_{{\rm l},-1}$ (resp. $\{(p_{+,1},p_{+,2})\} \subset S_{{\rm l},1}$) for suitable numbers $p_{+,1}>0$ and $p_{+,2}$, and $\{(0,a_{1,2})\}$ (resp. $\{(0,a_{-1,2})\}$), respectively. $C_{\rm s,+}$ is homeomorphic to $D^1$. In addition $C_{{\rm s},+} \bigcap (S_{{\rm l},-1} \bigcup S_{{\rm l},1})$ consists of all points of the boundary of $C_{{\rm s},+}$.
\item $C_{\rm s} \bigcap \{(0,x_2) \mid x_2 \in \mathbb{R}\}=\{(0,a_{1,2})\} {\rm (}\text{resp.
} \{(0,a_{-1,2})\}{\rm )}$.
\item $C_{\rm s}$ is homeomorphic to $D^1$.
\end{itemize}
\item In the case $u=m_{a}=m_{A}$ and $m_a=\frac{1}{1+{a_{-1,1}}^2}=\frac{1}{1+{a_{1,1}}^2}$, we need one connected, closed and compact subset $C_{\rm s} \subset {\overline{D_{S_a,a_{-1},a_{1}}}}^{{\mathbb{R}}^2}$ of ${\overline{D_{S_a,a_{-1},a_{1}}}}^{{\mathbb{R}}^n}$, as follows.
\begin{itemize}
\item $C_{\rm s,-}:=C_{\rm s} \bigcap \{(x_1,x_2) \mid x_1 \leq 0\}$ is a connected, closed and compact subset of ${\overline{D_{S_a,a_{-1},a_{1}}}}^{{\mathbb{R}}^n}$, and the restriction of the projection ${\pi}_{2,1,2}:{\mathbb{R}}^2 \rightarrow \mathbb{R}$ to $C_{{\rm s},-}$ is an injective real-valued function, with the intersections $C_{{\rm s},-} \bigcap S_{{\rm l},-1}$ and $C_{{\rm s},-} \bigcap S_{{\rm l},1}$ being one-point sets $\{(0,a_{-1,2})\}$, and $\{(0,a_{1,2})\}$, respectively, and consisting of all points of the boundary of $C_{{\rm s},-}$. 
\item $C_{\rm s,+}:=C_{\rm s} \bigcap \{(x_1,x_2) \mid x_1 \geq 0\}$ is a connected, closed and compact subset of ${\overline{D_{S_a,a_{-1},a_{1}}}}^{{\mathbb{R}}^n}$, and the restriction of the projection ${\pi}_{2,1,2}:{\mathbb{R}}^2 \rightarrow \mathbb{R}$ to $C_{{\rm s},+}$ is an injective real-valued function, with the intersections $C_{{\rm s},+} \bigcap S_{{\rm l},-1}$ and $C_{{\rm s},+} \bigcap S_{{\rm l},1}$ being one-point sets $\{(0,a_{-1,2})\}$, and $\{(0,a_{1,2})\}$, respectively, and consisting of all points of the boundary of $C_{{\rm s},+}$.
\item $C_{\rm s} \bigcap \{(0,x_2) \mid x_2 \in \mathbb{R}\}$ is equal to the union of the boundaries of $C_{\rm s,-}$ and $C_{\rm s,+}$.
\item $C_{\rm s}$ is homeomorphic to $S^1$.
\end{itemize}
\item In the case $m_{a}<u<m_A$ with $m_{a} \neq m_{A}$ and $m_a=\frac{1}{1+{a_{-1,1}}^2}$ ($m_a=\frac{1}{1+{a_{1,1}}^2}$), we need one connected, closed and compact subset $C_{\rm s} \subset {\overline{D_{S_a,a_{-1},a_{1}}}}^{{\mathbb{R}}^2}$ of ${\overline{D_{S_a,a_{-1},a_{1}}}}^{{\mathbb{R}}^n}$, as follows.
\begin{itemize}
\item The intersection $C_{\rm s} \bigcap S_{{\rm l},-1}$ ($C_{\rm s} \bigcap S_{{\rm l},1}$) is empty.
\item The intersection $C_{\rm s} \bigcap S_{{\rm l},1}$ (resp. $C_{\rm s} \bigcap S_{{\rm l},-1}$) is a two-point set\\
$\{(p_{m_a,m_A,-},p_{m_a,m_A,1}),(p_{m_a,m_A,+},p_{m_a,m_A,2})\}$ for suitable numbers $p_{m_a,m_A,-}<0$ and $p_{m_a,m_A,+}>0$, and $C_{\rm s}$ contains the unique point of the form $(0,x_2)$, denoted by $(0,p_{m_a,m_A})$.
\item The restriction of the projection ${\pi}_{2,1,2}:{\mathbb{R}}^2 \rightarrow \mathbb{R}$ to $C_{\rm s}$ is an injective real-valued function on the closure of each connected component of $C_{\rm s}-\{(0,p_{m_a,m_A})\}$ where the closure is taken in $C_{\rm s}$.
\item $C_{\rm s}$ is homeomorphic to $D^1$.

\end{itemize}
\end{itemize}
STEP  4-2-2 The case ${{{\pi}_{m+3,1} {\mid}_{X_{D_{S_a,a_{-1},a_{1}},m_1,m_2}}}^{\prime,{\rm P}}}^{-1}(u)$ ($m_A \leq u \leq 1$). \\
There exists some difference for the case $m_A<u<1$, between the cases (\ref{thm:4.2}) and (\ref{thm:4.3}).
\begin{itemize}
\item This is the case of (\ref{thm:4.2}). In this case, we do not need to consider the case $m_A<u<1$.
\item This is the case of (\ref{thm:4.3}). In the case $m_A<u<1$, we need one connected, closed and compact subset $C_{\rm s} \subset {\overline{D_{S_a,a_{-1},a_{1}}}}^{{\mathbb{R}}^2}$ of ${\overline{D_{S_a,a_{-1},a_{1}}}}^{{\mathbb{R}}^n}$, as follows.
\begin{itemize}
\item The intersection $C_{\rm s} \bigcap S_{{\rm l},-1}$ ($C_{\rm s} \bigcap S_{{\rm l},1}$) is empty.
\item The intersection $C_{\rm s} \bigcap S_{{\rm l},1}$ (resp. $C_{\rm s} \bigcap S_{{\rm l},-1}$) is a two-point set\\
$\{(p_{m_a,m_A,-},p_{m_a,m_A,1}),(p_{m_a,m_A,+},p_{m_a,m_A,2})\}$ for suitable numbers $p_{m_a,m_A,-}<0$ and $p_{m_a,m_A,+}>0$, and $C_{\rm s}$ contains exactly two points of the form $(0,x_2)$, denoted by $(0,p_{m_a,m_A,0,1})$ and $(0,p_{m_a,m_A.0,2})$ respectively.
\item The restriction of the projection ${\pi}_{2,1,2}:{\mathbb{R}}^2 \rightarrow \mathbb{R}$ to $C_{\rm s}$ is an injective real-valued function on the closure of each connected component of $C_{\rm s}-\{(0,p_{m_a,m_A,0,1}),(0,p_{m_a,m_A,0,2})\}$ where the closure is taken in $C_{\rm s}$. $C_{\rm s}$ is homeomorphic to $D^1$.

\end{itemize}
\end{itemize}

Throughout the present proof, the curves $C_{{\rm s},1}$, $C_{{\rm s},2}$, and $C_{\rm s}$, are smooth, densely, by fundamental arguments on existence of smooth curves (e.g. curve selection lemma in real analytic category in this scene and refer to \cite{milnor3}).  
Reeb digraphs of functions ${{\pi}_{m+3,1} {\mid}_{X_{D_{S_a,a_{-1},a_{1}},m_1,m_2}}}^{\prime,{\rm P}}$ are recovered from these arguments. Level sets of the functions are also recovered from them. For understanding the topologies and the differentiable structures of the smooth manifolds appearing as these level sets of the functions, we apply some classical facts from the so-called h-cobordism or s-cobordism theorem in the smooth category in dimensions at least $5$. See \cite{milnor2} again for this. 
We use the classification theorem of $5$-dimensional closed and simply-connected manifolds in the smooth category, the piecewise smooth (PL) one, and the topology category, due to Barden (\cite{barden}), to know that a topological manifold homeomorphic to $S^2 \times S^3$ is diffeomorphic to it. 
Related to this, for h-cobordism and s-cobordism theory in the topology category, refer to \cite{freedmanquinn}.

This completes the proof.
\end{proof}
We give several Remarks.
\begin{Rem}
\label{rem:2}
This is related to Remark \ref{rem:1}. In Theorem \ref{thm:4}, the function $\bar{{\pi}_{m+3,1} {\mid}_{X_{D_{S_a,a_{-1},a_{1}},m_1,m_2}}}^{\prime,{\rm P}}:\overrightarrow{R_{{{\pi}_{m+3,1} {\mid}_{X_{D_{S_a,a_{-1},a_{1}},m_1,m_2}}}^{\prime,{\rm P}}}} \rightarrow \mathbb{R}$ on the Reeb digraph $\overrightarrow{R_{{{\pi}_{m+3,1} {\mid}_{X_{D_{S_a,a_{-1},a_{1}},m_1,m_2}}}^{\prime,{\rm P}}}}$ can be represented as the composition of the embedding into ${\mathbb{R}}^2$ with ${\pi}_{2,1}$ if and only if  $m_1 \neq 1$. See also Figure \ref{fig:2} and examples are depicted.
\end{Rem} 
\begin{figure}
	\includegraphics[width=40mm, height=40mm]{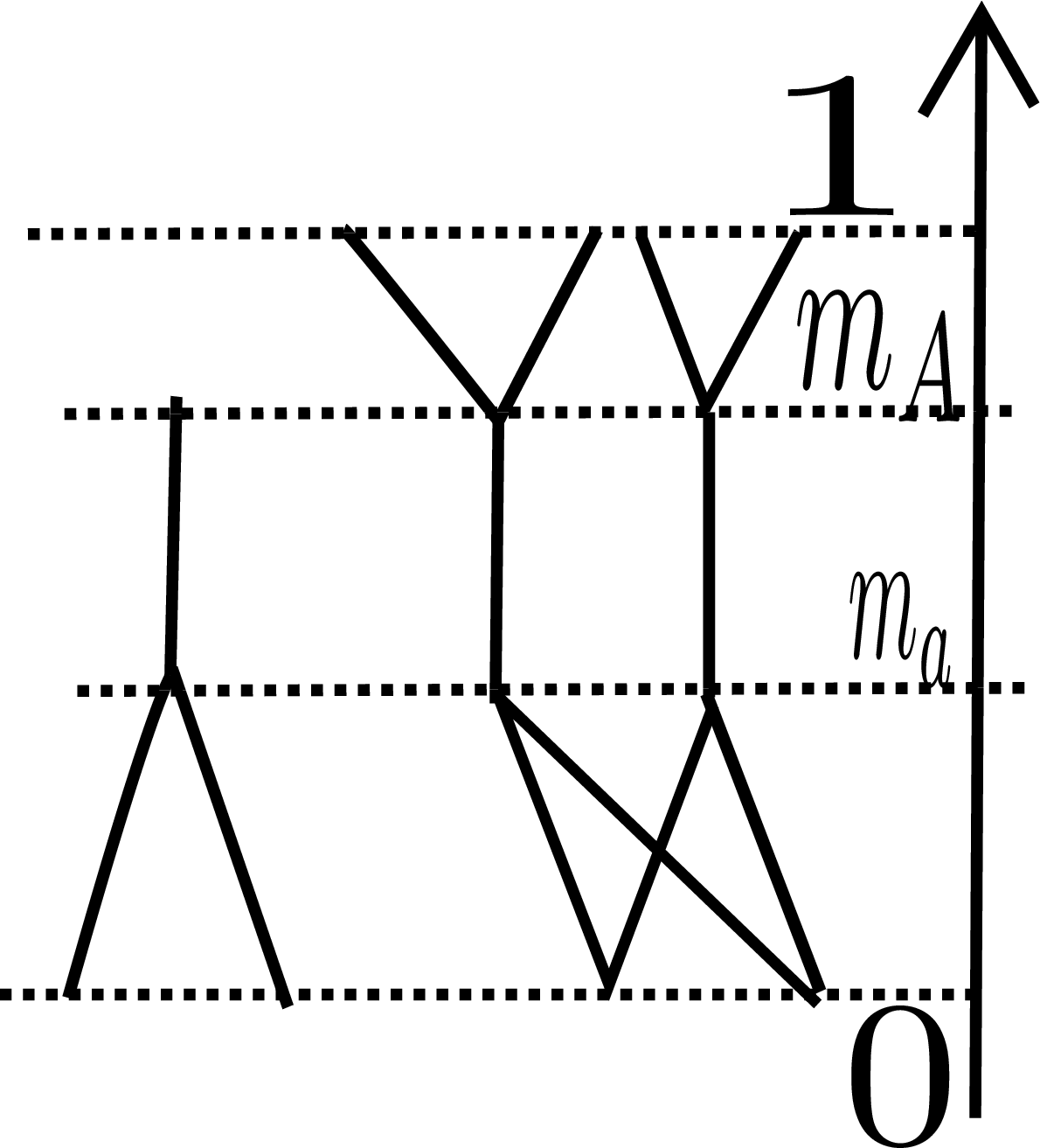}
	\caption{The functions $\bar{{\pi}_{m+3,1} {\mid}_{X_{D_{S_a,a_{-1},a_{1}},m_1,m_2}}}^{\prime,{\rm P}}:\overrightarrow{R_{{{\pi}_{m+3,1} {\mid}_{X_{D_{S_a,a_{-1},a_{1}},m_1,m_2}}}^{\prime,{\rm P}}}} \rightarrow \mathbb{R}$ on the Reeb digraphs
$\overrightarrow{R_{{{\pi}_{m+3,1} {\mid}_{X_{D_{S_a,a_{-1},a_{1}},m_1,m_2}}}^{\prime,{\rm P}}}}$, with the relations ${{\pi}_{m+3,1} {\mid}_{X_{D_{S_a,a_{-1},a_{1}},m_1,m_2}}}^{\prime,{\rm P}}=\bar{{\pi}_{m+3,1} {\mid}_{X_{D_{S_a,a_{-1},a_{1}},m_1,m_2}}}^{\prime,{\rm P}} \circ q_{{{\pi}_{m+3,1} {\mid}_{X_{D_{S_a,a_{-1},a_{1}},m_1,m_2}}}^{\prime,{\rm P}}}$ in Theorem \ref{thm:4}. The left case shows a case of Theorem \ref{thm:4} (\ref{thm:4.2}) with $m_1 \geq 2$ and the right case shows a case of (\ref{thm:4.3}) with $(m_1,m_2)=(1,1)$.}
\label{fig:2}
\end{figure}
\begin{Rem}
	\label{rem:3}
			If the curves $C_{{\rm s},1}$, $C_{{\rm s},2}$, and $C_{\rm s}$ are smooth globally, then "homeomorphic" in level sets ${{{\pi}_{m+3,1} {\mid}_{X_{D_{S_a,a_{-1},a_{1}},m_1,m_2}}}^{\prime,{\rm P}}}^{-1}(u)$ of the functions ${{\pi}_{m+3,1} {\mid}_{X_{D_{S_a,a_{-1},a_{1}},m_1,m_2}}}^{\prime,{\rm P}}$ of Theorem \ref{thm:4} can be replaced by "diffeomorphic". However, the author does not know the answer.
\end{Rem}
\begin{Rem}	
\label{rem:4}
Considerable parts of Theorem \ref{thm:3} and \ref{thm:4} hold for essentially general cases of \cite{kitazawa8}. Readers may check as a kind of exercises, where we do not assume non-trivial arguments of it, in main ingredients of the present paper.
\end{Rem}
\section{Conflict of interest, data availability, and so on.}
 The author is a researcher at Osaka Central Advanced Mathematical Institute (OCAMI researcher) and the institute is funded by MEXT Promotion of Distinctive Joint Research Center Program JPMXP0723833165. Note that he is not employed by the institute or the projects there. He thanks all there for the hospitality. 

This is submitted as a replacement of \cite{kitazawa8}, based on directions by https://arxiv.org/ in the initial submission.
 
No data other than the present file is generated, related to the present paper. Non-trivial arguments in formally unpublished preprints are not assumed. Referring to them to some extent contains no problem.


\end{document}